\documentclass[11pt]{article}%
\usepackage{amsmath}
\usepackage{amsfonts}
\usepackage{amssymb}
\usepackage{amsthm}
\usepackage{graphicx}%
\providecommand{\U}[1]{\protect\rule{.1in}{.1in}}
\newtheorem{theorem}{Theorem}[section]

\newtheorem{corollary}[theorem]{Corollary}
\newtheorem{lemma}[theorem]{Lemma}
\newtheorem{proposition}[theorem]{Proposition}
\theoremstyle{definition}
\newtheorem{definition}[theorem]{Definition}
\theoremstyle{remark}
\newtheorem{example}[theorem]{Example}
\newtheorem{remark}[theorem]{Remark}
\numberwithin{equation}{section}
\begin{document}

\title{Spectral parameter power series for perturbed Bessel equations}
\author{Ra\'{u}l Castillo-P\'{e}rez$^{1}$, Vladislav V. Kravchenko$^{2}$ and Sergii M.
Torba$^{2}$\\{\small $^{1}$SEPI, ESIME Zacatenco, Instituto Polit\'{e}cnico Nacional, Av.
IPN S/N, C.P. 07738, D.F. Mexico}\\{\small $^{2}$Departamento de Matem\'{a}ticas, CINVESTAV del IPN, Unidad
Quer\'{e}taro, Libramiento Norponiente \#2000,}\\{\small Fracc. Real de Juriquilla, Quer\'{e}taro, Qro. C.P. 76230 MEXICO}\\e-mail: rcastillo@ipn.mx, vkravchenko@math.cinvestav.edu.mx,\\storba@math.cinvestav.edu.mx}
\maketitle

\begin{abstract}
A spectral parameter power series (SPPS) representation for regular solutions
of singular Bessel type Sturm-Liouville equations with complex coefficients is
obtained as well as an SPPS representation for the (entire) characteristic
function of the corresponding spectral problem on a finite interval. It is
proved that the set of zeros of the characteristic function coincides with the set of all eigenvalues
of the Sturm-Liouville problem. Based on the SPPS representation a new mapping
property of the transmutation operator for the considered perturbed Bessel
operator is obtained, and a new numerical method for solving corresponding
spectral problems is developed. The range of applicability of the method
includes complex coefficients, complex spectrum and equations in which the
spectral parameter stands at a first order linear differential operator. On a
set of known test problems we show that the developed numerical method based
on the SPPS representation is highly competitive in comparison to the best
available solvers such as SLEIGN2, MATSLISE and some other codes and give an
example of an exactly solvable test problem admitting complex eigenvalues to
which the mentioned solvers are not applicable meanwhile the SPPS method
delivers excellent numerical results.

\end{abstract}

\section{Introduction}

In the present work the equation
\begin{equation}
-u^{\prime\prime}+\left(  \frac{l(l+1)}{x^{2}}+q(x)\right)  u=\lambda
\bigl(r_{1}(x)u^{\prime}+r_{0}(x)u\bigr),\qquad x\in(0,a], \label{I1}%
\end{equation}
is studied, where $l$ is a real number, $l\geq-\frac{1}{2}$, $q$ is a
complex-valued continuous function on $(0,a]$ satisfying a growth bound
$|q(x)|\le Cx^{\alpha}$ at the origin for some $\alpha>-2$, $r_{0,1}\in
C\left[  0,a\right]  $ are complex valued functions, and $\lambda$ is a
(complex) spectral parameter. Denote $L=-\frac{d^{2}}{dx^{2}}+\frac
{l(l+1)}{x^{2}}+q(x)$. Equations of the form (\ref{I1}) appear naturally in
many real-world applications after a separation of variables and therefore
have received considerable attention (see, e.g., \cite{BoumenirChanane},
\cite{Chebli1994}, \cite{Guillot 1988}, \cite{KosTesh2011}, \cite{Weidmann}).
In preceding publications equation (\ref{I1}) was studied under more
restrictive conditions, typically for $q$ and $r_{0}$ being real-valued and
$r_{1}\equiv0$. The approach developed in this work does not imply such
restrictions and serves both for qualitative study of solutions and spectral
problems, as well as for related numerical computation.

The main component in the developed approach is a spectral parameter power
series (SPPS) representation for the regular solution of (\ref{I1}) obtained
under the condition that the auxiliary equation $Lu_{0}=0$ possesses a regular
solution which does not have zeros on $[0,a]$ except at $x=0$. The SPPS
representation for solutions of nonsingular linear differential equations and
its applications in corresponding scattering and spectral problems were studied
in \cite{CKKO2009}, \cite{CKOR}, \cite{ErbeMertPeterson2012}, \cite{KKB2013},
\cite{KiraRosu2010}, \cite{Khmelnytskaya Serroukh 2013 MMAS},
\cite{Khmelnytskaya Torchinska 2010}, \cite{KrCV08}, \cite{APFT},
\cite{KKR2012}, \cite{KrPorter2010}, \cite{KrVelasco 2011}, \cite{Rabinovich
et al 2013 MMAS} and some other papers. Here, in Section \ref{Sect2} we obtain
an analogous result for the perturbed Bessel equation (\ref{I1}). The
construction and the existence of the required particular solution are
addressed in Section \ref{Sect3}. For example, when $q(x)\geq0$, $x\in(0,a]$
such nonvanishing on $(0,a]$ solution exists. We give an analytic
representation for it together with an estimate. Let us emphasize that, in
general, $u_{0}$ is allowed to be a complex-valued function, and the existence of such $u_{0}$ for a complex valued $q$ is an open problem.

Under the assumption that $u_{0}$ exists we obtain a dispersion
(characteristic) relation for the Sturm-Liouville problem for (\ref{I1}) on
$[0,a]$, $a<\infty$ (Theorem \ref{ThmSpectralProblem}). Namely, we construct
an entire function $\Phi(\lambda)$ in the form of a Taylor series such that the
set of its zeros coincides with the set of all eigenvalues of the
Sturm-Liouville problem. This immediately implies the discreteness of the
spectrum and offers an efficient method for its computation.

In practical applications of the SPPS method it is often convenient to
consider not only series with the centre in the origin $\lambda=0$ but also
series with the centre at $\lambda=\lambda_{0}$ where $\lambda_{0}$ is some
complex number. In Section \ref{SectSpShift} we show that this spectral
parameter shifting technique is applicable to equation (\ref{I1}) and give
necessary details.

The SPPS representation allows us to obtain a result on mapping properties of
the transmutation operator corresponding to the operator $L$, which was
studied, e.g., in \cite{Chebli1994}, \cite{Stashevskaya} and \cite{Volk}. We
show in Section \ref{SectTransmut} how the transmutation operator acts on certain powers of the independent
variable. In the case of nonsingular Schr\"{o}dinger operators a result of this kind
allowed us to advance in the construction of the transmutation operator itself
\cite{KT Transmut} and had applications in constructing complete systems of
solutions for some partial differential equations \cite{CKM}, \cite{CKT},
\cite{KKTT}.

In Section \ref{SectNumeric} we explain the numerical implementation of the
developed SPPS method for solving Sturm-Liouville problems for (\ref{I1}).
First, we consider several known test problems, and comparing the obtained
results with the results obtained by the best available codes, as SLEIGN2,
MATSLISE and some others, we show that our method is highly competitive and
gives better or at least comparable results on test problems to which other
codes are applicable. Second, we consider an example which involves a
different from zero $r_{1}$ in (\ref{I1}) and a complex spectrum. Meanwhile
SLEIGN2 and MATSLISE are not applicable to problems admitting complex
eigenvalues, our method delivers results which are in excellent agreement with
the exact data.

\section{Construction of the bounded solution of a perturbed Bessel
equation\label{Sect2}}

Consider a perturbed Bessel operator (also known as a spherical
Schr\"{o}dinger operator)
\begin{equation}
L=-\frac{d^{2}}{dx^{2}}+\frac{l(l+1)}{x^{2}}+q(x),\qquad l\geq-\frac{1}%
{2},\ x\in(0,a], \label{OpSingularSL}%
\end{equation}
where the potential $q$ is (in general) a complex-valued continuous function
on $(0,a]$ satisfying the growth condition in the origin
\begin{equation}
q(x)=O(x^{\alpha}),\quad x\rightarrow0\quad\text{for some }\alpha>-2.
\label{EqGrowthQ}%
\end{equation}
Note that we understand the $O$-notation in the sense of inequality,
i.e., there exist a neighborhood $(0,\varepsilon]$ of zero and a constant
$C>0$ such that $|q(x)|\leq Cx^{\alpha}$ for all $x\in(0,\varepsilon]$.

If $l\neq0$ or $q\not \in L^{1}(0,a]$, the left endpoint is singular. Despite
of that, the equation $Lu=0$ possesses a solution $\phi(x)$ which is bounded
at $x=0$ and satisfies the following asymptotics at $x=0$
\begin{align}
\phi(x)  &  \sim x^{l+1},\quad x\rightarrow0,\label{SolAsymptotic}\\
\phi^{\prime}(x)  &  \sim(l+1)x^{l},\quad x\rightarrow0,
\label{DerSolAsymptotic}%
\end{align}
see, e.g., \cite[Lemma 3.2]{KosTesh2011} for the real-valued potential $q$. We
show the explicit construction of the solution with this asymptotics at
zero for the general case of complex potentials in Section \ref{SectPartSol}
meanwhile in Section \ref{SectSpectrProb} we show that such solution is unique.

Together with $L$ consider a linear differential operator
\begin{equation}
Ru=r_{0}u+r_{1}u^{\prime} \label{OpRHS}%
\end{equation}
of order at most one, where $r_{0,1}\in C[0,a]$ are complex-valued functions,
and consider the following differential equation involving a spectral
parameter $\lambda$
\begin{equation}
Lu=\lambda Ru\qquad\text{or}\qquad-u^{\prime\prime}+\left(  \frac
{l(l+1)}{x^{2}}+q(x)\right)  u=\lambda\bigl(r_{1}(x)u^{\prime}+r_{0}(x)u\bigr).
\label{MainEq}%
\end{equation}

In order to construct a spectral parameter power series representation of a
non-singular in zero solution of \eqref{MainEq} assume that there exists a
non-vanishing on $(0,a]$ complex-valued solution $u_{0}$ of the equation
\begin{equation}
-u_{0}^{\prime\prime}+\left(  \frac{l(l+1)}{x^{2}}+q(x)\right)  u_{0}=0
\label{PartSolEq}%
\end{equation}
satisfying together with its first derivative the asymptotic relations
\eqref{SolAsymptotic} and \eqref{DerSolAsymptotic}. Let us define the
following system of recursive integrals
\begin{equation}%
\begin{split}
\widetilde{X}^{(0)}  &  \equiv1,\qquad\widetilde{X}^{(-1)}\equiv0,\\
\widetilde{X}^{(n)}(x)  &  =%
\begin{cases}
\displaystyle\int_{0}^{x}u_{0}(t)R\bigl[u_{0}(t)\widetilde{X}^{(n-1)}%
(t)\bigr]\,dt, & \text{if }n\text{ is odd},\\
-\displaystyle\int_{0}^{x}\frac{\widetilde{X}^{(n-1)}(t)}{u_{0}^{2}(t)}\,dt, &
\text{if }n\text{ is even}.
\end{cases}
\end{split}
\label{Xtilde}%
\end{equation}
We keep the notation $\widetilde{X}$ for consistency with the notations
from other publications on the SPPS method, see, e.g., \cite{KrPorter2010},
\cite{KKR2012}, \cite{KT Obzor}. Note that for an odd $n$ we have%
\[%
\begin{split}
R\bigl[u_{0}\widetilde{X}^{(n-1)}\bigr]  &  =\Bigl(r_{1}\frac{d}%
{dx}+r_{0}\Bigr)\bigl(u_{0}\widetilde{X}^{(n-1)}\bigr)\\
&  =r_{1}u_{0}^{\prime}\widetilde{X}^{(n-1)}-r_{1}u_{0}\frac{\widetilde
{X}^{(n-2)}}{u_{0}^{2}}+r_{0}u_{0}\widetilde{X}^{(n-1)}=R[u_{0}]\widetilde
{X}^{(n-1)}-\frac{r_{1}}{u_{0}}\widetilde{X}^{(n-2)}.
\end{split}
\]
Hence we can write \eqref{Xtilde} in a different form, which does not require
differentiation of the functions $\widetilde{X}^{(n)}$,
\begin{equation}
\widetilde{X}^{(n)}(x)=%
\begin{cases}
\displaystyle\int_{0}^{x}\bigl(u_{0}(t)R[u_{0}](t)\widetilde{X}^{(n-1)}%
(t)-r_{1}(t)\widetilde{X}^{(n-2)}(t)\bigr)\,dt, & \text{if }n\text{ is odd},\\
-\displaystyle\int_{0}^{x}\frac{\widetilde{X}^{(n-1)}(t)}{u_{0}^{2}(t)}\,dt, &
\text{if }n\text{ is even}.
\end{cases}
\label{XtildeAlt}%
\end{equation}

The following lemma establishes that all the involved integrals in
\eqref{XtildeAlt} are well defined and provides some estimates for the
functions $\widetilde{X}^{(n)}$.

\begin{lemma}
\label{LemmaXtildeEstimate} Let \eqref{PartSolEq} admit a solution $u_{0}\in
C[0,a]\cap C^{2}(0,a]$ (in general, complex-valued) which does not have other
zeros on $[0,a]$ except at $x=0$ and satisfies the asymptotic relations
\eqref{SolAsymptotic} and \eqref{DerSolAsymptotic}. Then the system of
functions $\bigl\{\widetilde{X}^{(n)}\bigr\}_{n=0}^{\infty}$ is well defined
by \eqref{Xtilde} or \eqref{XtildeAlt} and the functions $\widetilde{X}^{(n)}$
satisfy the inequalities
\begin{equation}
\bigl|\widetilde{X}^{(2n)}(x)\bigr|\leq\frac{C^{2n}x^{n}}{(2l+2)_{n}}%
,\qquad\bigl|\widetilde{X}^{(2n+1)}(x)\bigr|\leq\frac{(n+1)C^{2n+1}%
x^{2(l+1)+n}}{(2l+2)_{n+1}}, \label{XtildeEstimate}%
\end{equation}
where $(x)_{n}=\frac{\Gamma(x+n)}{\Gamma(x)}=x(x+1)\cdot\ldots\cdot(x+n-1)$ is
the Pochhammer symbol, $C=\max\{1,C_{1},C_{2},C_{3}\}$ and the constants
$C_{1}$, $C_{2}$ and $C_{3}$ are such that for any $t\in(0,a]$ the following
inequalities hold
\begin{equation}
\bigl|u_{0}(t)R[u_{0}](t)\bigr|\leq C_{1}t^{2l+1},\qquad\left\vert \frac
{1}{u_{0}^{2}(t)}\right\vert \leq C_{2}t^{-2l-2},\qquad|r_{1}(t)|\leq C_{3}.
\label{WeightsEstimate}%
\end{equation}

\end{lemma}

\begin{remark}
The constants $C_{1}$ and $C_{2}$ in Lemma \ref{LemmaXtildeEstimate} exist due
to the fact that $u_{0}$ is a non-vanishing function possessing asymptotics
\eqref{SolAsymptotic} and \eqref{DerSolAsymptotic}.
\end{remark}

\begin{proof}
The proof is by induction. Indeed, for $n=0$ we have $|\widetilde X^{(0)}(x)|\le 1$ and
\begin{equation*}
\bigl| \widetilde X^{(1)}(x)\bigr|\le \int_0^x \bigl| u_0(t) R[u_0](t)\bigr|\,dt \le \int_0^x C_1 t^{2l+1}\,dt\le\frac{Cx^{2l+2}}{2l+2}.
\end{equation*}
Assuming that the statement is true for some $n=k$, for $n=k+1$ we obtain
\begin{equation*}
\bigl|\widetilde X^{(2(k+1))}(x)\bigr| \le \int_0^x\left|\frac{\widetilde X^{(2k+1)}(t)}{u_0^2(t)}\right|\,dt
\le \int_0^x\frac{C_2}{t^{2l+2}}\cdot \frac{(k+1)C^{2k+1}t^{2(l+1)+k}}{(2l+2)_{k+1}}\,dt
\le \frac{C^{2k+2}x^{k+1}}{(2l+2)_{k+1}}
\end{equation*}
and
\begin{equation*}
\begin{split}
\bigl|\widetilde X^{(2(k+1)+1)}(x)\bigr| & \le \int_0^x\bigl| u_0(t)R[u_0](t)\widetilde X^{(2k+2)}(t)\bigr|\,dt+\int_0^x\bigl| r_1(t)\widetilde X^{(2k+1)}(t)\bigr|\,dt \\
& \le \int_0^x C_1 t^{2l+1}\cdot \frac{C^{2k+2}t^{k+1}}{(2l+2)_{k+1}}\,dt+\int_0^xC_3\cdot \frac{(k+1)C^{2k+1}t^{2(l+1)+k}}{(2l+2)_{k+1}}\,dt\\
& \le \frac{C^{2k+3}x^{2l+2+k+1}}{(2l+2+k+1)\cdot (2l+2)_{k+1}}+\frac{C^{2k+2}\cdot(k+1)x^{2l+2+k+1}}{(2l+2+k+1)\cdot (2l+2)_{k+1}}\\
&\le \frac{(k+2)C^{2k+3}x^{2(l+1)+k+1}}{ (2l+2)_{k+2}}.
\end{split}
\end{equation*}
Note that the exponents of the powers of $t$ in all the involved integrands are non-negative, hence all the recursive integrals are well defined.
\end{proof}

In the particular case when $r_{1}\equiv0$, i.e., the right hand side of
\eqref{MainEq} does not depend on the derivative of $u$, the estimates of
Lemma \ref{LemmaXtildeEstimate} can be improved, and the following statement
is valid.

\begin{lemma}
\label{LemmaXtildeEstimateZeroR1} Under the conditions of Lemma
\ref{LemmaXtildeEstimate} assume additionally that $r_{1}\equiv0$. Then the
functions $\widetilde{X}^{(n)}$ defined by \eqref{Xtilde} or \eqref{XtildeAlt}
satisfy the  inequalities
\begin{equation}
\bigl|\widetilde{X}^{(2n)}(x)\bigr|\leq\frac{C^{2n}x^{2n}}{2^{2n}%
n!(l+3/2)_{n}},\qquad\bigl|\widetilde{X}^{(2n+1)}(x)\bigr|\leq\frac
{C^{2n+1}x^{2n+1+2(l+1)}}{2^{2n+1}n!(l+3/2)_{n+1}},
\label{XtildeEstimateZeroR1}%
\end{equation}
where $(x)_{n}$ is the Pochhammer symbol and $C=\max\{C_{1},C_{2}\}$, where
the constants $C_{1}$ and $C_{2}$ are such that for any $t\in(0,a]$ the following
inequalities hold
\begin{equation}
\bigl|r_{0}(t)u_{0}^{2}(t)\bigr|\leq C_{1}t^{2l+2},\qquad\left\vert \frac
{1}{u_{0}^{2}(t)}\right\vert \leq C_{2}t^{-2l-2}.
\label{WeightsEstimateZeroR1}%
\end{equation}

\end{lemma}

The proof is analogous to that of Lemma \ref{LemmaXtildeEstimate}.

The following theorem presents the spectral parameter power series (SPPS)
representation of a bounded solution of equation \eqref{MainEq}.

\begin{theorem}
\label{ThmSingularSPPS} Let \eqref{PartSolEq} admit a solution $u_{0}\in
C[0,a]\cap C^{2}(0,a]$ (in general, complex-valued) which does not have other
zeros on $[0,a]$ except at $x=0$ and satisfies the asymptotic relations
\eqref{SolAsymptotic} and \eqref{DerSolAsymptotic}. Then for any $\lambda
\in\mathbb{C}$ the function
\begin{equation}
u=u_{0}\sum_{k=0}^{\infty}\lambda^{k}\widetilde{X}^{(2k)} \label{SPPSsol}%
\end{equation}
is a solution of \eqref{MainEq} belonging to $C[0,a]\cap C^{2}(0,a]$ and the
series converges uniformly on $[0,a]$. The first derivative of $u$ is given
by
\begin{equation}
u^{\prime}=\frac{u_{0}^{\prime}}{u_{0}}u-\frac{1}{u_{0}}\sum_{k=1}^{\infty
}\lambda^{k}\widetilde{X}^{(2k-1)}, \label{SPPSsolDer}%
\end{equation}
and the series for the first and the second derivatives converge uniformly on
an arbitrary compact $K\subset(0,a]$.
\end{theorem}

\begin{proof}
Formally differentiating the series \eqref{SPPSsol} twice with the aid of \eqref{XtildeAlt} we obtain that $u'$ should be given by \eqref{SPPSsolDer} and $u''$ (after simplification) by $\frac{u_0''}{u_0}u-\lambda r_0u - \lambda r_1u'$. By Lemma \ref{LemmaXtildeEstimate} all the involved series converge uniformly on an arbitrary compact $K\subset(0,a]$, hence the formal derivatives coincide with the usual ones. Since by \eqref{OpRHS} and \eqref{PartSolEq}
\begin{equation*}
\frac{u_0''}{u_0}u-\lambda r_0u-\lambda r_1u'=\left(\frac{l(l+1)}{x^2}+q(x)\right)u-\lambda R[u],
\end{equation*}
$u$ is indeed a solution of equation \eqref{MainEq}. The relations \eqref{SolAsymptotic} and \eqref{DerSolAsymptotic} follow from the corresponding asymptotics of $u_0$ because by Lemma \ref{LemmaXtildeEstimate} we have $\widetilde X^{(2k)}(x)=o(1)$ and $\frac{1}{u_0}\widetilde X^{(2k-1)}=o(x^l)$ for $k\ge 1$.
\end{proof}

\begin{remark}
For a regular Sturm-Liouville problem the existence and the construction of the
required non-vanishing solution $u_{0}$ present no difficulty since the
equation possesses two linearly independent real-valued solutions $u_{1}$ and $u_{2}$
whose zeros alternate and one may choose $u_0=u_{1}+iu_{2}$ as such solution. For
the singular equation under consideration there is only one solution satisfying the
asymptotic conditions \eqref{SolAsymptotic} and \eqref{DerSolAsymptotic},
see Theorem \ref{ThmSpectralProblem}. Corollary \ref{CorrNonVanishingPS}
establishes that such non-vanishing solution $u_{0}$ exists in the case when
$q(x)\ge0$, $x\in(0,a]$, and in Remark \ref{RmkBddQ} we show that a modified
SPPS representation is always possible in the case when $q$ is real valued and
bounded from below and $r_{1}$ is real valued.
\end{remark}

\begin{remark}
For SPPS representations for solutions of nonsingular Sturm-Liouville
equations we refer to \cite{KrCV08}, \cite{APFT} and \cite{KrPorter2010}. They
have been applied in a number of papers to different scattering and spectral
problems (see references in the Introduction). For the perturbed Bessel
equation, in the case of a real valued potential $q$, $r_{1}\equiv0$ and
$r_{0}\equiv1$ an SPPS representation was obtained and used in
\cite{KosTesh2011} but without formulas for constructing or estimating the coefficients
$\widetilde{X}^{(2k)}$.
\end{remark}

For practical applications the partial sums of the series \eqref{SPPSsol} are
of the main interest. Based on Lemmas \ref{LemmaXtildeEstimate} and
\ref{LemmaXtildeEstimateZeroR1} the following corollary provides estimates for
the difference between the exact solution and the approximate one defined as a
partial sum of the series \eqref{SPPSsol}. The difference is estimated in
terms of the remainders of Taylor series of two special functions.

\begin{corollary}
\label{CorrSPPSDiscrepancy} Under the conditions of Theorem
\ref{ThmSingularSPPS} consider $u_{N}=u_{0}\sum_{k=0}^{N}\lambda^{k}%
\widetilde{X}^{(2k)}$, note that for $N=0$ the right-hand side is equal to
$u_{0}$. Then
\begin{equation}
|u(x)-u_{N}(x)|\leq\max_{t\in\lbrack0,x]}|u_{0}(t)|\cdot\sum_{k=N+1}^{\infty
}\frac{|\lambda|^{k}C^{2k}x^{k}}{(2l+2)_{k}}\leq\max_{t\in\lbrack0,x]}%
|u_{0}(t)|\cdot\biggl|e^{C^{2}|\lambda|x}-\sum_{k=0}^{N}\frac{(C^{2}%
|\lambda|x)^{k}}{k!}\biggr|, \label{EqSPPSsolDiscrepancy}%
\end{equation}
where the constant $C$ is defined in Lemma \ref{LemmaXtildeEstimate}.

Moreover, in the particular case $r_{1}\equiv0$ the following estimate holds
\begin{equation}%
\begin{split}
|u(x)-u_{N}(x)|  &  \leq\max_{t\in\lbrack0,x]}|u_{0}(t)|\cdot\sum
_{k=N+1}^{\infty}\frac{|\lambda|^{k}C^{2k}x^{2k}}{2^{2k}k!(l+3/2)_{k}}\\
&  \leq\max_{t\in\lbrack0,x]}|u_{0}(t)|\cdot\biggl|\frac{2^{l+1/2}%
\Gamma(l+\frac32)}{(|\lambda|^{1/2}Cx) ^{l+1/2}}I_{l+1/2}(|\lambda
|^{1/2}Cx)-\sum_{k=0}^{N}\frac{(C^{2}|\lambda|x^{2})^{k}}{2^{2k}k!(l+3/2)_{k}%
}\biggr|,
\end{split}
\label{EqSPPSsolDiscrepancyZeroR1}%
\end{equation}
where $\Gamma(l+3/2)$ is the Gamma function, $I_{l+1/2}(x)$ is the modified Bessel function of the first kind and the
constant $C$ is defined in Lemma \ref{LemmaXtildeEstimateZeroR1}.
\end{corollary}

For the proof one should simply compare the majorizing terms corresponding to
the even indices in \eqref{XtildeEstimate} and \eqref{XtildeEstimateZeroR1}
with the Taylor expansions for the exponential and the Bessel functions
appearing in \eqref{EqSPPSsolDiscrepancy} and \eqref{EqSPPSsolDiscrepancyZeroR1}.

\begin{example}
\label{ExampleBessel} Consider the Bessel equation
\[
-\frac{d^{2}}{dx^{2}}u+\frac{l(l+1)}{x^{2}}u=\lambda u,
\]
with $l\geq-1/2$ and $x\in(0,a]$. The regular solution of this equation
satisfying \eqref{SolAsymptotic} and \eqref{DerSolAsymptotic} is given by the
formula%
\begin{equation}
\label{ExampleBesselSol}u_{l}(x,\lambda)=\Gamma(l+3/2)2^{l+1/2}\lambda
^{-\frac{2l+1}{4}}\cdot\sqrt{x}J_{l+1/2}(\sqrt{\lambda}x),
\end{equation}
where $J_{l+1/2}$ is the Bessel
function of the first kind. This solution may be represented as a power series
in terms of the parameter $\lambda$,%
\begin{equation}
u_{l}(x,\lambda)=x^{l+1}\sum_{k=0}^{\infty}\frac{(-1)^{k}x^{2k}}%
{4^{k}k!(l+3/2)_{k}}\lambda^{k}, \label{ExampleBesselSeries}%
\end{equation}
see \cite{KosTesh2011}.

In order to apply Theorem \ref{ThmSingularSPPS} consider $u_{0}(x)=x^{l+1}$ as
the non-vanishing on $(0,a]$ solution of the equation $-u^{\prime\prime}%
+\frac{l(l+1)}{x^{2}}u=0$, satisfying \eqref{SolAsymptotic} and
\eqref{DerSolAsymptotic}. It is easy to verify that choosing such $u_{0}$ we
obtain from \eqref{XtildeAlt}
\[
\widetilde{X}^{(2n)}(x)=\frac{(-1)^{n}x^{2n}}{4^{n}n!(l+3/2)_{n}},
\]
i.e., exactly the coefficients from \eqref{ExampleBesselSeries}.
\end{example}

\section{Construction of the particular solution $u_{0}$\label{Sect3}}

\label{SectPartSol} In this section we explain how to construct a particular
solution of equation \eqref{PartSolEq} satisfying asymptotics
\eqref{SolAsymptotic}, \eqref{DerSolAsymptotic} and present some sufficient
conditions for this solution to be non-vanishing for $x>0$.

In order to construct an SPPS representation for the particular solution we
rewrite equation \eqref{PartSolEq} in the form
\begin{equation}
y^{\prime\prime}-\frac{l(l+1)}{x^{2}}y=q(x)y. \label{PartSolEqTransf}%
\end{equation}
The equation
\begin{equation}
y_{0}^{\prime\prime}-\frac{l(l+1)}{x^{2}}y_{0}=0 \label{PartSolEqPartEq}%
\end{equation}
possesses two solutions
\[
y_{1}(x)=x^{-l}\qquad\text{and}\qquad y_{2}(x)=x^{l+1}.
\]
For $l\geq-1/2$ the second solution is regular and satisfies
\eqref{SolAsymptotic} and \eqref{DerSolAsymptotic}. Since the potential $q$
may be singular in the origin, we cannot apply Theorem \ref{ThmSingularSPPS}
directly. However we may construct the system of recursive integrals in the
same way as in \eqref{XtildeAlt} and only have to show the convergence of the
integrals and obtain some estimates justifying the SPPS representation.

Consider the following system of recursive integrals
\begin{equation}
\label{YtildePS}%
\begin{split}
\widetilde Y^{(0)}  &  \equiv1,\\
\widetilde Y^{(n)}(x)  &  =
\begin{cases}
\displaystyle\int_{0}^{x}\widetilde Y^{(n-1)}(t) t^{2(l+1)} q(t)\,dt, &
\text{for odd }n,\\
\displaystyle\int_{0}^{x}\widetilde Y^{(n-1)}(t) t^{-2(l+1)}\,dt, & \text{for
even }n.
\end{cases}
\end{split}
\end{equation}

Note that since the potential $q\in C(0,a]$ satisfies condition
\eqref{EqGrowthQ} for some $\alpha> -2$ there exists a constant $C>0$ such
that
\begin{equation}
\label{EstimateQGlobal}|q(x)|\le Cx^{\alpha}\qquad\text{for all }x\in(0,a].
\end{equation}

\begin{lemma}
\label{LemmaYtildeEstimate} Suppose that the complex-valued potential $q\in
C(0,a]$ satisfies inequality \eqref{EstimateQGlobal} for some $C>0$ and
$\alpha>-2$. Then the functions $\widetilde{Y}^{(n)}$ are well defined by
\eqref{YtildePS} and the following estimates hold
\begin{align}
\bigl|\widetilde{Y}^{(2n)}(x)\bigr|  &  \leq\frac{C^{n}x^{n(2+\alpha)}%
}{(2+\alpha)^{2n}n!\bigl(\frac{2l+1}{2+\alpha}+1\bigr)_{n}}%
,\label{YtildeEvenEstimate}\\
\bigl|\widetilde{Y}^{(2n-1)}(x)\bigr|  &  \leq\frac{C^{n}x^{2l+1+n(2+\alpha)}%
}{(2+\alpha)^{2n-1}n!\bigl(\frac{2l+1}{2+\alpha}+1\bigr)_{n}},\qquad
x\in(0,a], \label{YtildeOddEstimate}%
\end{align}
where $(x)_{n}$ is the Pochhammer symbol.
\end{lemma}

\begin{proof}The proof can be performed by induction, similarly to the proof of Lemma \ref{LemmaXtildeEstimate}.
The only difference is that it is possible that an exponent of the power of $t$ under the integral sign be negative,
however it is always strictly greater than $-1$. Hence all the involved integrals exist.\end{proof}

\begin{proposition}
\label{PropPartSol} Suppose that the complex-valued potential $q\in C(0,a]$
satisfies inequality \eqref{EstimateQGlobal} for some $C>0$ and $\alpha>-2$.
Then the function
\begin{equation}
u_{0}(x)=x^{l+1}\sum_{k=0}^{\infty}\widetilde{Y}^{(2k)}(x), \label{FnPartSol}%
\end{equation}
where the functions $\widetilde{Y}^{(2k)}$ are defined by \eqref{YtildePS}, is
a particular solution of equation \eqref{PartSolEqTransf} on $(0,a]$
satisfying asymptotics \eqref{SolAsymptotic}, \eqref{DerSolAsymptotic}.
Moreover, $u_{0}$ satisfies the following estimate for any $x\in(0,a]$
\begin{equation}
|u_{0}(x)|\leq\Gamma\left(  \frac{2l+1}{2+\alpha}+1\right)  (2+\alpha
)^{\frac{2l+1}{2+\alpha}}C^{-\frac{2l+1}{4+2\alpha}}\sqrt{x}I_{\frac
{2l+1}{2+\alpha}}\left(  \frac{2\sqrt{Cx^{2+\alpha}}}{2+\alpha}\right)  .
\label{FnPartSolEstimate}%
\end{equation}

\end{proposition}

\begin{proof}
Due to \eqref{YtildePS} the first and the second derivatives of $u_0$ are given by the expressions
\begin{equation}\label{FnPartSolDer}
u_0' = (l+1)x^l\sum_{k=0}^\infty\widetilde Y^{(2k)}+x^{-l-1}\sum_{k=1}^\infty \widetilde Y^{(2k-1)}
\end{equation}
and
\begin{equation*}
\begin{split}
u_0'' & = (l+1)l x^{l-1}\sum_{k=0}^\infty\widetilde Y^{(2k)}+(l+1)x^{-l-2}\sum_{k=1}^\infty\widetilde Y^{(2k-1)} \\
& -(l+1)x^{-l-2}\sum_{k=1}^\infty\widetilde Y^{(2k-1)} + x^{l+1}q\sum_{k=1}^\infty\widetilde Y^{(2k-2)}  = \frac{(l+1)l}{x^2}u_0 + qu_0.
\end{split}
\end{equation*}
The uniform convergence of all the involved series on an arbitrary compact $K\subset(0,a]$
and hence the possibility of termwise differentiation, follows from Lemma \ref{LemmaYtildeEstimate}.
The asymptotic relations \eqref{SolAsymptotic} and \eqref{DerSolAsymptotic} for the function $u_0$ follow
from the estimates $\widetilde Y^{(2k)} = o(1)$ and $\widetilde Y^{(2k-1)} = o(x^{2l+1})$, $x\to 0$ for $k\ge 1$,
see \eqref{YtildeEvenEstimate} and \eqref{YtildeOddEstimate}. Inequality \eqref{FnPartSolEstimate} follows from the
series representation of the modified Bessel functions of the first kind, see, e.g., \cite{AbramovitzStegun}.
\end{proof}

The following corollary provides a sufficient condition for the particular
solution constructed in Proposition \ref{PropPartSol} to be non-vanishing.

\begin{corollary}
\label{CorrNonVanishingPS} Under the conditions of Proposition
\ref{PropPartSol} assume additionally that $q(x)\ge0$, $x\in(0,a]$. Then
$u_{0}(x)\ge x^{l+1}$ for any $x\in(0,a]$.
\end{corollary}

\begin{proof}We obtain from \eqref{YtildePS} and the condition $q(x)\ge 0$ that $\widetilde Y^{(n)}(x)\ge 0$, $n\ge 1$. Hence, $u_0(x) = x^{l+1}\bigl( \widetilde Y^{(0)}+\sum_{k=1}^\infty \widetilde Y^{(2k)}\bigr)\ge x^{l+1}$.
\end{proof}

\section{Spectral problems\label{SectSpectrProb}}

The classical formulation (see, e.g., \cite{KosTesh2011}, \cite{Zettl}) of a
spectral problem for a Sturm-Liouville equation of the form $Lu=\lambda ru$
with $L$ from \eqref{OpSingularSL} consists in finding the values of the spectral parameter for which there exists a solution $u(x;\lambda)$ continuous at $x=0$ and satisfying the following conditions. When $l>1/2$,
\begin{equation}
u(0;\lambda)=0 \label{EqBC1}%
\end{equation}
and
\begin{equation}
\beta u(a;\lambda)+\gamma u^{\prime}(a;\lambda)=0 \label{EqBC2}%
\end{equation}
for some $\beta,\gamma\in\mathbb{C}$ such that $|\beta|+|\gamma|\neq0$.

When $l\in\lbrack-1/2,1/2)$ and since the second linearly independent
solution is also square-integrable, an additional boundary condition
\begin{equation}
\lim_{x\rightarrow0}x^{l}\bigl((l+1)u(x;\lambda)-xu^{\prime}(x;\lambda
)\bigr)=0, \label{EqBC3}%
\end{equation}
is imposed (see, e.g., \cite{KosTesh2011}).

Despite the fact that the right-hand side \eqref{OpRHS} of the spectral
equation \eqref{MainEq} may contain the derivative of the unknown function and
hence not fit into the basic Sturm-Liouville scheme, we consider
the same spectral problem \eqref{EqBC1}--\eqref{EqBC3} for equation \eqref{MainEq}. The following statement gives us a characteristic equation of the spectral problem under the condition that an appropriate non-vanishing solution $u_0$ of \eqref{PartSolEq} exists.

\begin{theorem}
\label{ThmSpectralProblem} Let \eqref{PartSolEq} admit a solution $u_{0}\in
C[0,a]\cap C^{2}(0,a]$ (in general, complex-valued) which does not have other
zeros on $[0,a]$ except at $x=0$ and satisfies the asymptotic relations
\eqref{SolAsymptotic} and \eqref{DerSolAsymptotic}. Then the eigenvalues of
the problem \eqref{MainEq}, \eqref{EqBC1}--\eqref{EqBC3} coincide with zeros
of the entire function
\begin{equation}
\Phi(\lambda)=\bigl(\beta u_{0}(a)+\gamma u_{0}^{\prime}(a)\bigr)\sum
_{k=0}^{\infty}\lambda^{k}\widetilde{X}^{(2k)}(a)-\frac{\gamma}{u_{0}(a)}%
\sum_{k=1}^{\infty}\lambda^{k}\widetilde{X}^{(2k-1)}(a), \label{CharEq}%
\end{equation}
where the functions $\widetilde{X}^{(k)}$ are defined by \eqref{Xtilde} or \eqref{XtildeAlt}.
\end{theorem}

\begin{proof}
Under the condition of the theorem, the function $u$ defined by \eqref{SPPSsol} is a solution of \eqref{MainEq} satisfying the boundary conditions \eqref{EqBC1} and \eqref{EqBC3}. The boundary condition \eqref{EqBC2} for the function $u$ coincides with $\Phi(\lambda)=0$, and $\Phi(\lambda)$ is an entire function by Theorem \ref{ThmSingularSPPS}.
It is left to show that there are no more eigenvalues. For that it is sufficient to show that the second linearly independent solution $u_2$ of \eqref{MainEq} does not satisfy either \eqref{EqBC1} or \eqref{EqBC3}.
We rewrite \eqref{MainEq} in the form
\begin{equation}\label{MainEq2}
-u''-\lambda r_1 u'+\Bigl(\frac{l(l+1)}{x^2} + q - \lambda r_0\Bigr) u = 0
\end{equation}
and introduce the function
\begin{equation*}
p(x)=e^{-\lambda\int_0^x r_1(t)\,dt}.
\end{equation*}
Note that
\begin{equation}\label{AsymptP}
\lim_{x\to 0} p(x)=1.
\end{equation}
Since the solution $u$ defined by \eqref{SPPSsol} satisfies the asymptotic condition \eqref{SolAsymptotic},
there exists a constant $b=b(\lambda)>0$ such that $u(x)\ne 0$ for all $x\in(0,b]$. Then the second linearly independent
solution is given by the Liouville formula \cite[Chap. XI]{Hartman}, \cite{RajStoj2007}
\begin{equation*}
u_2(x) = -u(x)\int_x^b\frac{p(t)}{u^2(t)}\,dt,\qquad x\in(0,b].
\end{equation*}
Assume first that $l>-1/2$. It follows from \eqref{AsymptP}, asymptotics \eqref{SolAsymptotic}, \eqref{DerSolAsymptotic}
and L'Hospital's rule that
\begin{equation}\label{Sol2Asympt}
u_2(x)\sim -\frac{x^{-l}}{2l+1},\qquad x\to 0.
\end{equation}
Hence for $l\ge 0$ the second solution $u_2$ does not satisfy the boundary condition \eqref{EqBC1}. Let now $-1/2<l<0$.
Recall that by Abel's identity \cite[Chap. XI]{Hartman} the Wronskian of $u$ and $u_2$ has the form
\begin{equation}\label{Wronskian}
W=uu_2'-u'u_2 = p.
\end{equation}
Hence using \eqref{SolAsymptotic}, \eqref{DerSolAsymptotic}, \eqref{AsymptP} and \eqref{Sol2Asympt} we obtain
\begin{equation}\label{DelSol2Asympt}
u_2'(x)\sim \frac{l}{2l+1} x^{-l-1},\qquad x\to 0,
\end{equation}
and thus observe that $u_2$ cannot satisfy the boundary condition \eqref{EqBC3}.

For $l=-1/2$ similar reasoning shows that
\begin{equation}\label{L1/2Sol2Asympt}
u_2(x)\sim -\sqrt{x}\ln x,\qquad x\to 0.
\end{equation}
We substitute the expression for $u_2'$ obtained from \eqref{Wronskian} into \eqref{EqBC3} and obtain
\begin{equation*}
x^{-1/2}\Bigl(\frac 12u_2 - xu_2'\Bigr) = -\sqrt{x}\frac{p}{u} + \frac{u_2}{u}\cdot\frac{u-2xu'}{2\sqrt{x}}.
\end{equation*}
For the first term in this expression,  from \eqref{AsymptP} and \eqref{SolAsymptotic} we have
\begin{equation*}
\lim_{x\to 0} \sqrt{x}\frac{p(x)}{u(x)} = 1,
\end{equation*}
hence to prove that $u_2$ does not satisfy the boundary condition \eqref{EqBC3} it is sufficient to show that the second term is $o(1)$ as $x\to 0$. Due to the asymptotic relations \eqref{SolAsymptotic} and \eqref{L1/2Sol2Asympt} it is sufficient to show that $u-2xu' = O(x^{1/2+\varepsilon})$ for some $\varepsilon>0$. Taking into account \eqref{SPPSsolDer} we have
\begin{equation*}
u-2xu'=\frac u{u_0}\bigl(u_0-2x u_0'\bigr)+\frac{2x}{u_0}\sum_{k=1}^\infty \lambda^k \widetilde X^{(2k-1)}.
\end{equation*}
As can be seen from \eqref{XtildeEstimate}, $\frac{2x}{u_0}\sum_{k=1}^\infty \lambda^k \widetilde X^{(2k-1)}=o(x)$, hence only the first term is relevant. Since $u\sim u_0\sim \sqrt x$ as $x\to 0$, it is sufficient to prove that $u_0-2xu_0' = O(x^{1/2+\varepsilon})$.
Similarly to \eqref{L1/2Sol2Asympt} we obtain that the general solution of equation \eqref{PartSolEq} can be represented as $c_1\tilde u_1+c_2\tilde u_2$, where $\tilde u_1$ is given by \eqref{FnPartSol} and satisfies the asymptotic condition \eqref{SolAsymptotic} and $\tilde u_2$ satisfies the asymptotic condition $\tilde u_2\sim \sqrt{x}\ln x$, $x\to 0$. Since $u_0\sim \sqrt{x}$ by the statement of the theorem, it is necessarily of the form \eqref{FnPartSol}. Now using \eqref{FnPartSol}, \eqref{FnPartSolDer} and \eqref{YtildeOddEstimate} we obtain
\begin{equation*}
u_0-2xu_0' = \sqrt x\sum_{k=0}^\infty \widetilde Y^{(2k)} - \frac {2x}{2\sqrt x}\sum_{k=0}^\infty \widetilde Y^{(2k)} - 2\sqrt x\sum_{k=1}^\infty \widetilde Y^{(2k-1)} = O(x^{1/2+2+\alpha}),
\end{equation*}
where $\alpha$ participates in \eqref{EqGrowthQ}, and this finishes the proof for $l=-1/2$.
\end{proof}

\section{Spectral shift technique\label{SectSpShift}}

The SPPS representation given in Theorem \ref{ThmSingularSPPS} is based on a
particular solution of equation \eqref{MainEq} for $\lambda=0$. In
\cite{KrPorter2010} it was mentioned that for a classic Sturm-Liouville
equation it is also possible to construct the SPPS representation of a general
solution starting from a non-vanishing particular solution for some
$\lambda=\lambda_{0}$. Such procedure is called spectral shift and has already
proven its usefulness for numerical applications \cite{KrPorter2010},
\cite{KKB2013}.

We show that a spectral shift technique may be applied to equation
\eqref{MainEq}. Let $\lambda_{0}$ be a fixed complex number. We rewrite
\eqref{MainEq} in the form
\begin{equation}
-u^{\prime\prime}-\lambda_{0}r_{1}u^{\prime}+\left(  \frac{l(l+1)}{x^{2}%
}+q-\lambda_{0}r_{0}\right)  u=\widetilde{\lambda}\bigl(r_{1}u^{\prime
}+r_{0}u\bigr), \label{MainEqSpShift}%
\end{equation}
where $\widetilde{\lambda}:=\lambda-\lambda_{0}$. Suppose that $u_{0}$ is a
solution of the equation
\begin{equation}
L_{0}u:=-u^{\prime\prime}-\lambda_{0}r_{1}u^{\prime}+\left(  \frac
{l(l+1)}{x^{2}}+q-\lambda_{0}r_{0}\right)  u=0 \label{PartSolEqSpShift}%
\end{equation}
such that $u_{0}$ does not vanish on $(0,a]$. Note that $u_{0}$ is a
particular solution of \eqref{MainEq} for $\lambda=\lambda_{0}$. Then the
operator $L_{0}$ admits the following P\'{o}lya factorization
\cite{Polya1924}
\begin{equation}
L_{0}u=-\frac{1}{pu_{0}}\frac{d}{dx}pu_{0}^{2}\frac{d}{dx}\frac{u}{u_{0}},
\label{Lfactorization}%
\end{equation}
where
\begin{equation}
p(x)=e^{\lambda_{0}\int_{0}^{x}r_{1}(s)\,ds}. \label{PolyaP}%
\end{equation}
Based on the factorization \eqref{Lfactorization} we introduce the following
system of recursive integrals
\begin{equation}%
\begin{split}
\widetilde{Z}^{(0)}  &  \equiv1,\qquad\widetilde{Z}^{(-1)}\equiv0,\\
\widetilde{Z}^{(n)}(x)  &  =%
\begin{cases}
\displaystyle\int_{0}^{x}p(t)u_{0}(t)R\bigl[u_{0}(t)\widetilde{Z}%
^{(n-1)}(t)\bigr]\,dt, & \text{if }n\text{ is odd},\\
-\displaystyle\int_{0}^{x}\frac{\widetilde{Z}^{(n-1)}(t)}{p(t)u_{0}^{2}%
(t)}\,dt, & \text{if }n\text{ is even}.
\end{cases}
\end{split}
\label{Ztilde}%
\end{equation}
Similarly to \eqref{XtildeAlt} the recurrent relation \eqref{Ztilde} can be
rewritten as
\begin{equation}
\widetilde{Z}^{(n)}(x)=%
\begin{cases}
\displaystyle\int_{0}^{x}\bigl(p(t)u_{0}(t)R[u_{0}](t)\widetilde{Z}%
^{(n-1)}(t)-r_{1}(t)\widetilde{Z}^{(n-2)}(t)\bigr)\,dt, & \text{if }n\text{ is
odd},\\
-\displaystyle\int_{0}^{x}\frac{\widetilde{Z}^{(n-1)}(t)}{p(t)u_{0}^{2}%
(t)}\,dt, & \text{if }n\text{ is even}.
\end{cases}
\label{ZtildeAlt}%
\end{equation}

\begin{theorem}
\label{ThmSpectralShift} Let \eqref{PartSolEqSpShift} admit a solution
$u_{0}\in C[0,a]\cap C^{2}(0,a]$ (in general, complex-valued) which does not
have other zeros on $[0,a]$ except at $x=0$ and satisfies the asymptotic
relations \eqref{SolAsymptotic} and \eqref{DerSolAsymptotic}. Then for any
$\lambda\in\mathbb{C}$ the function
\begin{equation}
u=u_{0}\sum_{k=0}^{\infty}(\lambda-\lambda_{0})^{k}\widetilde{Z}^{(2k)}
\label{SPPSsolSpShift}%
\end{equation}
is a solution of \eqref{MainEq} belonging to $C[0,a]\cap C^{2}(0,a]$, and the
series \eqref{SPPSsolSpShift} converges uniformly on $[0,a]$. The series for
the first and the second derivatives converge uniformly on an arbitrary
compact $K\subset(0,a]$ and the first derivative has the form
\begin{equation}
u^{\prime}=\frac{u^{\prime}_{0}}{u_{0}} u - \frac{1}{pu_{0}}\sum_{k=1}%
^{\infty}(\lambda-\lambda_{0})^{k}\widetilde{Z}^{(2k-1)}= u_{0}^{\prime}%
+\sum_{k=1}^{\infty}(\lambda-\lambda_{0})^{k}\biggl(\widetilde{Z}^{(2k)}%
-\frac{\widetilde{Z}^{(2k-1)}}{pu_{0}}\biggr) . \label{SPPSdersolSpShift}%
\end{equation}

\end{theorem}

\begin{proof}
Note that the function $p$ given by \eqref{PolyaP} is continuous and non-vanishing on $[0,a]$. Similarly to the proof of Lemma \ref{LemmaXtildeEstimate} we see that the functions $\widetilde Z^{(n)}$ satisfy estimates \eqref{XtildeEstimate} with the constant
\begin{equation}\label{ConstantSpShift}
C = \max\bigl\{1, C_1, C_2, C_3\bigr\},
\end{equation}
where
\begin{equation*}
C_1 = \sup_{t\in(0,a]}\frac{\bigl| p(t)u_0(t) R[u_0](t)\bigr|}{t^{2l+1}},\qquad
C_2 = \max_{t\in[0,a]}\frac{t^{2l+2}}{\bigl|p(t)u_0^2(t)\bigr|},\qquad
C_3 = \max_{t\in[0,a]}\bigl|r_1(t)\bigr|.
\end{equation*}
Formal application of the operator $L_0$ to \eqref{SPPSsolSpShift} with the use of P\'{o}lya factorization \eqref{Lfactorization} and formula \eqref{Ztilde} shows that the function $u$ is a solution of equation \eqref{MainEqSpShift} and hence of \eqref{MainEq}. Estimates for $\widetilde Z^{(n)}$ justify the possibility of differentiation of the involved series and show that the function $u$ satisfies \eqref{SolAsymptotic} and \eqref{DerSolAsymptotic}.
\end{proof}

\begin{remark}
\label{RmkBddQ} Suppose that the functions $q$ and $r_{1}$ are real-valued and
that the potential $q$ is bounded from below, i.e., there exists a constant
$q_{0}\in\mathbb{R}$ such that
\[
q(x)\ge q_{0}\qquad\text{for all }x\in(0,a].
\]
Consider equation \eqref{PartSolEqSpShift} for $\lambda_{0}=q_{0}$. The
particular solution of this equation can be constructed similarly to Section
\ref{SectPartSol} using the generalization of formulas \eqref{YtildePS}
according to the factorization \eqref{Lfactorization}, cf., \eqref{Xtilde} and
\eqref{Ztilde}. Since the functions $p$ and $q-q_{0}$ are non-negative,
similarly to Corollary \ref{CorrNonVanishingPS} we deduce that the particular
solution $u_{0}$ in such case does not have other zeros on $[0,a]$ except at
$x=0$. Hence it is possible to construct the SPPS representation of the
bounded solution for any equation \eqref{MainEq} having a real-valued $r_{1}$
and a real-valued bounded from below $q$.
\end{remark}

\begin{remark}
In the case when $r_{1}\equiv0$, we do not need to introduce the new system of
functions $\widetilde Z^{(n)}$. The system of functions $\widetilde X^{(n)}$
given by \eqref{XtildeAlt} may be used directly in representation \eqref{SPPSsolSpShift}.
\end{remark}

\begin{remark}
For the difference $\max_{x\in\lbrack0,a]}|u(x)-u_{N}(x)|$, where $u_{N}%
=u_{0}\sum_{k=0}^{N}\lambda^{k}\widetilde{Z}^{(2k)}$ the estimate
\eqref{EqSPPSsolDiscrepancy} holds, where $C$ is given by \eqref{ConstantSpShift}.
\end{remark}

\section{Transmutation operators for perturbed Bessel operators}\label{SectTransmut}

We recall a general definition of a transmutation operator from \cite{KT
Obzor} which is a modification of the definition given by Levitan
\cite{LevitanInverse}. Let $E$ be a linear topological space and $E_{1}$ its
linear subspace (not necessarily closed). Let $A$ and $B$ be linear operators:
$E_{1}\rightarrow E$.

\begin{definition}
\label{DefTransmut} A linear invertible operator $T$ defined on the whole $E$
such that $E_{1}$ is invariant under the action of $T$ is called a
transmutation operator for the pair of operators $A$ and $B$ if it fulfills
the following two conditions.

\begin{enumerate}
\item Both the operator $T$ and its inverse $T^{-1}$ are continuous in $E$;

\item The following operator equality is valid
\[
AT=TB
\]
or which is the same
\[
A=TBT^{-1}.
\]

\end{enumerate}
\end{definition}

Very often in literature the transmutation operators (the term coined by  Delsarte and Lions \cite{DelsarteLions1956})  are called transformation operators.

In \cite{CKT} for the case of the transmutation operator $T$ corresponding to the pair of operators $A=-\frac
{d^{2}}{dx^{2}}+q(x)$ and $B=-\frac{d^{2}}{dx^{2}}$ a mapping property of $T$ was found. It establishes what is the result of action of $T$ on the powers of the independent variable. This mapping property already found many applications in the proofs of
completeness of infinite systems of solutions of some linear partial differential equations and in solving initial and
spectral problems, see \cite{CKM}, \cite{CKT}, \cite{KKTT}, \cite{KT MMET
2012}, \cite{KT Transmut}, \cite{KT NumTrans}. We present an analogue of
the aforementioned mapping property for the transmutation operator for the pair of
operators $A=-\frac{d^{2}}{dx^{2}}+q(x)+\frac{l(l+1)}{x^{2}}$ and
$B=-\frac{d^{2}}{dx^{2}}+\frac{l(l+1)}{x^{2}}$, constructed in
\cite{Stashevskaya}, \cite{Volk}.

We recall some results from \cite{Stashevskaya}, \cite{Volk}. Under certain
additional conditions on the potential $q$, discussed below, a bounded
solution of the equation
\begin{equation}
\label{EqSingVolk}-y^{\prime\prime}+\Bigl(q(x)+\frac{l(l+1)}{x^{2}}\Bigr)y =
\lambda y
\end{equation}
can be represented in the form
\begin{equation}
\label{TransmutSingular}y(x,\lambda) = j_{l+1/2}(x,\lambda)+\int_{0}^{x}
K(x,t) j_{l+1/2}(t,\lambda)\,dt,
\end{equation}
where $j_{l+1/2}(x,\lambda) = \sqrt{x\sqrt{\lambda}}J_{l+1/2}(x\sqrt{\lambda
})$ is a solution of the equation
\begin{equation}
\label{EqSingVolkWOpotential}-y^{\prime\prime}+\frac{l(l+1)}{x^{2}}y = \lambda
y
\end{equation}
and $J_{l+1/2}$ is the Bessel function of the first kind. The integral kernel
$K$ is the solution of the partial differential equation
\[
\frac{\partial^{2}K(x,t)}{\partial x^{2}}-\frac{l(l+1)}{x^{2}}K(x,t) -
q(x)K(x,t) = \frac{\partial^{2}K(x,t)}{\partial t^{2}}-\frac{l(l+1)}{t^{2}%
}K(x,t)
\]
satisfying the boundary conditions
\[
\frac{dK(x,x)}{dx}=\frac12 q(x)\quad\text{and}\quad\lim_{t\to0} K(x,t)\cdot
t^{l} = 0.
\]
Moreover, the integral kernel $K$ satisfies
\begin{equation}
\label{KinL2}\sup_{0\le x\le a}\int_{0}^{x} |K(x,t)|^{2}\,dt<\infty.
\end{equation}

We denote the operator defined by \eqref{TransmutSingular} as $\mathbf{T}$.
The existence of such operator was established in \cite{Volk} for the case
when $q$ is a continuous function on $[0,a]$ and in \cite{Stashevskaya} for
the case when $l$ is an integer and $q$ is a real-valued function satisfying the
condition $\int_{0}^{a} t^{m}|q(t)|\,dt<\infty$ for some $0<m<1/2$. It is
mentioned in \cite{Sitnik} that the results of \cite{Akhiezer1957} allow one to
extend the existence of the operator $\mathbf{T}$ onto arbitrary real values of
the parameter $l$.

The solution $y(x,\lambda)$ in \eqref{TransmutSingular} differs from the
solution $u(x,\lambda)$ satisfying the asymptotic condition \eqref{SolAsymptotic} by the
factor $\frac{\lambda^{(l+1)/2}}{2^{l+1/2}\Gamma(l+3/2)}$, see
\cite{Stashevskaya}. The series expansion of the function $j_{l+1/2}%
(x,\lambda)$ is
\begin{equation}
\label{jSeries}j_{l+1/2}(x,\lambda) = \sum_{k=0}^{\infty}c_{k} x^{2k+l+1}%
,\qquad\text{where } c_{k}=c_{k}(\lambda) = \frac{(-1)^{k}\lambda^{k+(l+1)/2}%
}{\Gamma(k+1)\Gamma(k+l+3/2)2^{2k+l+1/2}}.
\end{equation}
Similarly to \cite[Theorem 7]{CKT} we substitute
the functions $y(x,\lambda)$ and $j_{l+1/2}(x,\lambda)$ in \eqref{TransmutSingular} by their series expansions \eqref{SPPSsol} and \eqref{jSeries} and obtain
\begin{equation}
\label{EqForMappingProperty}%
\begin{split}
\frac{\lambda^{(l+1)/2}}{2^{l+1/2}\Gamma(l+\frac32)} u_{0}(x)\sum
_{k=0}^{\infty}\lambda^{k} \widetilde X^{(2k)}(x)=  &  \sum_{k=0}^{\infty
}c_{k}x^{2k+l+1} + \int_{0}^{x}\biggl( K(x,t)\sum_{k=0}^{\infty}%
c_{k}t^{2k+l+1}\biggr)\,dt\\
=  &  \sum_{k=0}^{\infty}c_{k}\biggl(x^{2k+l+1} + \int_{0}^{x}
K(x,t)t^{2k+l+1}\,dt\biggr)\\
=  &  \sum_{k=0}^{\infty}\frac{(-1)^{k}\lambda^{k+(l+1)/2}}{\Gamma
(k+1)\Gamma(k+l+\frac32)2^{2k+l+1/2}} \mathbf{T}[x^{2k+l+1}].
\end{split}
\end{equation}
The function $K(x,\cdot)$ is square-integrable on $[0,x]$, see \eqref{KinL2},
and the function $j_{l+1/2}(\cdot,\lambda)$ is the limit of the uniformly
convergent series \eqref{jSeries}, hence the possibility to change the
order of summation and integration in \eqref{EqForMappingProperty}
follows from the continuity of the scalar product in $L^{2}[0,a]$.

Since the equality in \eqref{EqForMappingProperty} holds for all $x$ and
$\lambda$, we finally obtain that
\begin{equation}
\label{MappingProperty}\mathbf{T}[x^{2k+l+1}] = (-1)^{k} 2^{2k} k!
\Bigl(l+\frac32\Bigr)_{k} u_{0}(x) \widetilde X^{(2k)}(x).
\end{equation}

\section{Numerical implementation and examples\label{SectNumeric}}

Based on the results of the previous sections we can formulate a numerical
method for solving spectral problems for perturbed Bessel equations.

\begin{enumerate}
\item Find a particular solution $u_{0}$ of equation \eqref{PartSolEq}
satisfying the asymptotic conditions \eqref{SolAsymptotic} and
\eqref{DerSolAsymptotic}. Note that due to the proof of Theorem
\ref{ThmSpectralProblem} it is sufficient to check that the solution satisfies
only the asymptotic condition \eqref{SolAsymptotic}. If an analytic
expression for the particular solution is unknown, one can use a numerical
approximation suggested by Proposition \ref{PropPartSol}.

\item Check that the particular solution $u_{0}$ obtained in step 1 is
non-vanishing for $x\in(0,a]$. Under the conditions of Corollary
\ref{CorrNonVanishingPS} it is always the case. If the particular solution has
zeros on $(0,a]$, the spectral shift technique described in Section
\ref{SectSpShift} can help, either directly as described in Remark
\ref{RmkBddQ} or by finding a suitable value of $\lambda_{0}$, complex in general.

\item Use partial sums of the series \eqref{SPPSsol} and \eqref{SPPSsolDer} (or,
correspondingly, \eqref{SPPSsolSpShift} and \eqref{SPPSdersolSpShift}) to
obtain a polynomial
\begin{equation}
\label{PhiN}\Phi_{N}(\lambda) = \bigl(\beta u_{0}(a)+\gamma u_{0}^{\prime
}(a)\bigr)+ \sum_{k=1}^{N}\lambda^{k}\left(  \bigl(\beta u_{0}(a)+\gamma
u_{0}^{\prime}(a)\bigr) \widetilde{X}^{(2k)}(a)-\frac{\gamma}{u_{0}%
(a)}\widetilde{X}^{(2k-1)}(a)\right)
\end{equation}
approximating the characteristic function \eqref{CharEq}.

\item Find zeros of the calculated polynomial $\Phi_{N}(\lambda)$.

\item To improve the accuracy of the eigenvalues located farther from the
point $\lambda_{0}$, the spectral parameter corresponding to the current
particular solution, perform one or several steps of the spectral shift
technique.

\end{enumerate}

Since the characteristic function $\Phi(\lambda)$ is analytic, the Rouch\'{e}
theorem from complex analysis, see, e.g., \cite{Conway}, provides the
stability of the numerical procedure. Indeed, let $\Gamma$ be an arbitrary
simple closed contour on which $\Phi$ does not vanish. Then if the absolute
error of approximation $|\Phi_{N}-\Phi|$ is less than $\min|\Phi_{N}|$ on
$\Gamma$, due to Rouch\'{e}'s theorem the functions $\Phi_{N}$ and $(\Phi
-\Phi_{N})+\Phi_{N} = \Phi$ possess the same number of zeros inside $\Gamma$.
Hence the procedure described above does not produce additional (or on the
contrary less) zeros whenever $\Phi_{N}$ approximates well enough the function
$\Phi$. We illustrate this point below, in Example \ref{Ex6}.

Before considering numerical examples let us explain how the numerical
implementation of the SPPS method was done in this work. All the calculations
were performed with the help of Matlab 2010 in the double precision machine
arithmetics. The formal powers $\widetilde X^{(n)}$ were calculated using the
Newton-Cottes 6 point integration formula of 7-th order, see, e.g.,
\cite{DavisRabinovich}, modified to perform indefinite integration. We choose
$M$ equally spaced points covering the segment of interest and apply the
integration formula to overlapping groups of six points. It is worth
mentioning that for large values of the parameter $l$ a special care should be
taken near the point $0$, because even small errors in the values of
$\widetilde X^{(2n+1)}$ after the division by $u_{0}^{2}\sim x^{2(l+1)}$ lead
to large errors in the computation of $\widetilde X^{(2n)}$ on the whole
interval $[0,a]$. To overcome this difficulty we change the values of
$\widetilde X^{(2n+1)}$ in several points near zero to their asymptotic
values. This strategy leads to a good accuracy. The computation of the first
100--200 formal powers proved to be a completely feasible task, and even for
$M$ being as large as several millions the computation time of the whole set of formal powers is within seconds. In
the presented numerical results we specify two parameters: $N$ is the degree
of the polynomial $\Phi_{N}$ in \eqref{PhiN}, i.e., the number of the calculated
formal powers is $2N$, and $M$ is the number of points taken on the considered
segment for the calculation of integrals.

Let us stress that the formal powers do not depend on the spectral parameter
and once calculated can be used for computing the solution and/or the
characteristic function $\Phi(\lambda)$ for thousands of different values of
the spectral parameter $\lambda$ without any additional significant
computation cost.

There exist several computer codes for the solution of singular
Sturm-Liouville problems. We compare our results with the results produced by
SLEIGN2 \cite{BEZ2} and MATSLISE \cite{LVV}. Both packages can reliably solve
a variety of spectral problems for regular and singular Sturm-Liouville
problems and are considered as basic comparison tools, the second package to our best knowledge is one of the most accurate. Both packages work with
double precision machine arithmetics and were used with the parameters for the
highest possible accuracy goals. It should be mentioned that the applicability
of both packages is restricted to the self-adjoint situation and they do not
permit neither complex coefficients nor a first order differential operator at
a spectral parameter. Thus, to illustrate the performance of the proposed
method in a situation when a comparison to the existing
software is impossible, in Example \ref{Ex6} we consider a problem for which an exact characteristic equation can be written down and compare the obtained numerical results to those calculated from the exact equation.

\subsection{Real spectrum}

\begin{example}
\label{Ex1}Our first numerical example is the Bessel equation, which we
already touched on in Example \ref{ExampleBessel}. Consider the
following spectral problem \cite[Example 2]{BEZ}
\begin{equation}
\label{BEZex2eqn}%
\begin{cases}
-y^{\prime\prime}+\frac{c}{x^{2}}y=\lambda y, & 0<x\leq1,\\
y(1,\lambda)=0. &
\end{cases}
\end{equation}
In this and in all other considered examples the solution $y(x,\lambda)$ should also satisfy the
boundary conditions \eqref{EqBC1} and \eqref{EqBC3} at 0. As follows from
\eqref{ExampleBesselSol}, the eigenvalues of the problem \eqref{BEZex2eqn}
coincide with zeros of the Bessel function $J_{\nu}(s)$, where
$s=\sqrt{\lambda}$ and $\nu=\sqrt{c+\frac14}$.

We compare the results delivered by the SPPS method to those from \cite[Example 2]{BEZ} for a
particular case $c=\frac5{16}$. For the SPPS representation \eqref{SPPSsol} the exact
particular solution $u_{0}=x^{5/4}$ was used. The results are presented in Table
\ref{Ex1Table1} together with the exact eigenvalues calculated as squares of
zeros of $J_{3/4}(s)$ and computed using the Matlab routine
\texttt{besselzero.m} by Greg von Winckel.

\begin{table}[ptb]
\centering
\begin{tabular}
[c]{cr@{}lr@{}lr@{}l}\hline
$n$ & \multicolumn{2}{c}{$\lambda_{n}$ (SPPS)} & \multicolumn{2}{c}{$\lambda
_{n}$ (SLEIGN2, \cite{BEZ})} & \multicolumn{2}{c}{$\lambda_{n}$ (Exact)}%
\\\hline
1 & 12. & 1871394680951 & 12. & 187139459 & 12. & 1871394680951\\
2 & 44. & 257559403500 & 44. & 257558912 & 44. & 257559403502\\
3 & 96. & 07160483898 & 96. & 071604502 & 96. & 07160483884\\
4 & 167. & 62571241787 & 167. & 625711908 & 167. & 62571242058\\
5 & 258. & 91930037169 & 258. & 919292439 & 258. & 91930035744\\
6 & 369. & 95220905860 & 369. & 952209262 & 369. & 95220926235\\
7 & 500. & 72440133519 & 500. & 724370471 & 500. & 72438147579\\
8 & 651. & 23517308180 & 651. & 235865279 & 651. & 23579210254\\
9 & 821. & 50506498326 & 821. & 486428982 & 821. & 48642898238\\
10 & 988. & 97560762340 & 1011. & 476285608 & 1011. & 47628560802\\\hline
\end{tabular}
\caption{The eigenvalues from Example \ref{Ex1} for $c=\frac5{16}$, calculated
with $N=40$ and $M=5\cdot10^{4}$.}%
\label{Ex1Table1}%
\end{table}

As can be seen from Table \ref{Ex1Table1}, the accuracy of the higher
eigenvalues calculated by the SPPS method decreases. To improve the accuracy the spectral shift
technique described in Section \ref{SectSpShift} is applied. We choose that the values of the parameter $\lambda_{0}$ change
along a line in the complex plane and are given by $\lambda_{0}%
^{(n)}=50n+2ni$, $n=1\ldots4000$. On each step we compute the solution in the
next point $\lambda_{0}^{(n+1)}$ and use this solution as a particular
solution for the next step. Since the spectral problem has a purely real
spectrum, among the zeros of the approximating polynomial the ones
with a small imaginary part are chosen and those which are closer to the number
$\mathop{\mathrm{Re}}\lambda_{0}^{(n)}$ than to any other of the numbers
$\mathop{\mathrm{Re}}\lambda_{0}^{(k)}$, $k\ne n$ are stored. The obtained results are
presented in Table \ref{Ex1Table2} together with the used values of
$\lambda_{0}$ for the spectral shift and the exact eigenvalues. In the first column of the table we combine the
results produced by the MATSLISE package with the exact eigenvalues because in this case all
the presented digits coincide.

As can be seen from Table \ref{Ex1Table2}, the spectral shift technique allows
us to significantly improve the obtained results for the higher eigenvalues. From now on we present only the results obtained with the help of the spectral shift technique and specify the implemented rules for choosing spectral shifts.
\end{example}

\begin{table}[ptb]
\centering
\begin{tabular}
[c]{cr@{}lr@{}lr@{+}lr@{}l}\hline
$n$ & \multicolumn{2}{c}{$\lambda_{n}$ (exact/MATSLISE)} &
\multicolumn{2}{c}{$\lambda_{n}$ (SPPS)} & \multicolumn{2}{c}{$\lambda_{0}$
used} & \multicolumn{2}{c}{$\lambda_{n}$ (SLEIGN2)}\\\hline
1 & 12. & 1871394680951 & 12. & 1871394680951 & 1 & $0.5i$ & 12. & 187139459\\
2 & 44. & 257559403502 & 44. & 257559403500 & 50 & $2i$ & 44. & 257558912\\
3 & 96. & 071604838843 & 96. & 071604838834 & 100 & $4i$ & 96. & 071604502\\
4 & 167. & 62571242058 & 167. & 62571242056 & 150 & $6i$ & 167. & 625711908\\
5 & 258. & 91930035744 & 258. & 91930035742 & 250 & $10i$ & 258. & 919292439\\
6 & 369. & 95220926235 & 369. & 95220926232 & 350 & $14i$ & 369. & 952209262\\
7 & 500. & 72438147579 & 500. & 72438147575 & 500 & $20i$ & 500. & 724370471\\
8 & 651. & 23579210254 & 651. & 23579210250 & 650 & $26i$ & 651. & 235865279\\
9 & 821. & 48642898238 & 821. & 48642898235 & 800 & $32i$ & 821. & 486428982\\
10 & 1011. & 47628560802 & 1011. & 47628560801 & 1000 & $40i$ & 1011. &
476285608\\
30 & 8956. & 5077203636 & 8956. & 5077203638 & 8950 & $358i$ & 8956. &
507721371\\
50 & 24797. & 222775294 & 24797. & 222775296 & 24800 & $992i$ & 24796. &
866878938\\
75 & 55701. & 421553437 & 55701. & 421553167 & 50000 & $2000i$ & 55701. &
328815911\\
100 & 98942. & 625835 & 98942. & 625812 & 100000 & $4000i$ & 98943. &
253862034\\\hline
\end{tabular}
\caption{The eigenvalues from Example \ref{Ex1} for $c=\frac5{16}$, calculated
with the use of the spectral shift technique with $N=40$ and $M=5\cdot10^{4}%
$.}%
\label{Ex1Table2}%
\end{table}

\begin{example}
\label{Ex2} The second example is the Boyd equation. Consider the following
spectral problem \cite[Example 3]{BEZ}
\begin{equation}
\label{BEZex3eqn}%
\begin{cases}
-y^{\prime\prime}-\frac{1}{x}y=\lambda y, & 0<x\leq1\\
y(1,\lambda)=0. &
\end{cases}
\end{equation}
This equation fits into the general scheme if one takes $l=0$. We take the
function $u_{0}=\sqrt x J_{1}(2\sqrt x)$ as a particular solution of
\eqref{PartSolEq} and calculate eigenvalues using the SPPS method with $N=40$
and $M=50000$. For problem \eqref{BEZex3eqn} the characteristic equation is
known and is given by
\begin{equation}
\label{Ex2CharEq}k M_{k,1/2}(2i\sqrt{\lambda}) = 0,
\end{equation}
where $M_{k,1/2}$ is the Whittaker function and $k=k(\lambda)=(2i\sqrt
{\lambda})^{-1}$, see \cite[Example 3]{BEZ}. In Table \ref{Ex2Table2} we
present the obtained results together with the exact eigenvalues calculated
from \eqref{Ex2CharEq} and the results obtained using SLEIGN2 and MATSLISE
software. Note that some eigenvalues caused problems for the MATSLISE package,
despite the excellent accuracy of the rest of the results. To simplify the reading, in the
presented numbers we truncated some digits which do not coincide with the correct ones. Relative errors of the first 50 eigenvalues obtained by
SPPS, MATSLISE and SLEIGN2 are presented on Figure \ref{BEZex3error}. We emphasize a relative stability of the numerical results delivered by the SPPS method in comparison to those computed by MATSLISE and SLEIGN2. We mention also that the
same problem was considered in \cite{AKKW} where less accurate results were
reported: $\lambda_{1}=7.37398502$, $\lambda_{2}=36.3360196$, $\lambda
_{3}=85.2925811$, $\lambda_{4}=154.098619$ and $\lambda_{5}=242.705545$.

\begin{table}[ptb]
\centering
\begin{tabular}
[c]{cr@{}lr@{}lr@{}lr@{}lr@{}l}\hline
$n$ & \multicolumn{2}{c}{$\lambda_{n}$ (SPPS)} & \multicolumn{2}{c}{$\lambda
_{n}$ (exact)} & \multicolumn{2}{c}{$\lambda_{n}$ (SLEIGN2)} &
\multicolumn{2}{c}{$\lambda_{n}$ (MATSLISE)} &  & \\\hline
1 & 7. & 3739850151752 & 7. & 3739850151751 & 7. & 37398499 & 7. &
3739850151751 &  & \\
2 & 36. & 3360195952325 & 36. & 3360195952318 & 36. & 33601959522 & 36. &
3360195952318 &  & \\
3 & 85. & 292582094149 & 85. & 292582094137 & 85. & 29258240 & 85. &
292582094137 &  & \\
4 & 154. & 098623739770 & 154. & 098623739767 & 154. & 098623739742 & 154. &
098623739767 &  & \\
5 & 242. & 705559362903 & 242. & 705559362911 & 242. & 705559362924 & 242. &
705559362911 &  & \\
6 & 351. & 091167129407 & 351. & 091167129418 & 351. & 09116712937 & 351. &
091167129418 &  & \\
8 & 627. & 155044324547 & 627. & 155044324564 & 627. & 155033 & 627. &
155044324564 &  & \\
10 & 982. & 239093680177 & 982. & 239093680188 & 982. & 239069 & 982. &
239093680188 &  & \\
13 & 1662. & 98063088581 & 1662. & 98063088578 & 1662. & 9806308832 & 1662. &
76 &  & \\
20 & 3942. & 42966385096 & 3942. & 42966385102 & 3942. & 18 & 3942. &
42966385102 &  & \\
28 & 7732. & 02180519217 & 7732. & 02180519214 & 7732. & 021805177 & 7729. &
4 &  & \\
30 & 8876. & 82700072940 & 8876. & 82700072941 & 8876. & 8270009 & 8876. &
82700072941 &  & \\
35 & 12084. & 29442705883 & 12084. & 29442705875 & 12084. & 263 & 12083. &
98 &  & \\
40 & 15785. & 2626475018 & 15785. & 2626475007 & 15784. & 3 & 15785. &
2626475007 &  & \\
50 & 24667. & 683593322 & 24667. & 683593313 & 24667. & 683593398 & 24662. &
5 &  & \\\hline
\end{tabular}
\caption{The eigenvalues from Example \ref{Ex2}, calculated with the use of
the spectral shift technique with $N=40$, $M=50000$. The spectral shift is given by $\lambda_{0}^{(n)}=50n+(2n+0.5)i$.}
\label{Ex2Table2}%
\end{table}

\begin{figure}[ptb]
\centering
\includegraphics[bb=119 301 491 489,
height=2.35in,
width=4.65in
]
{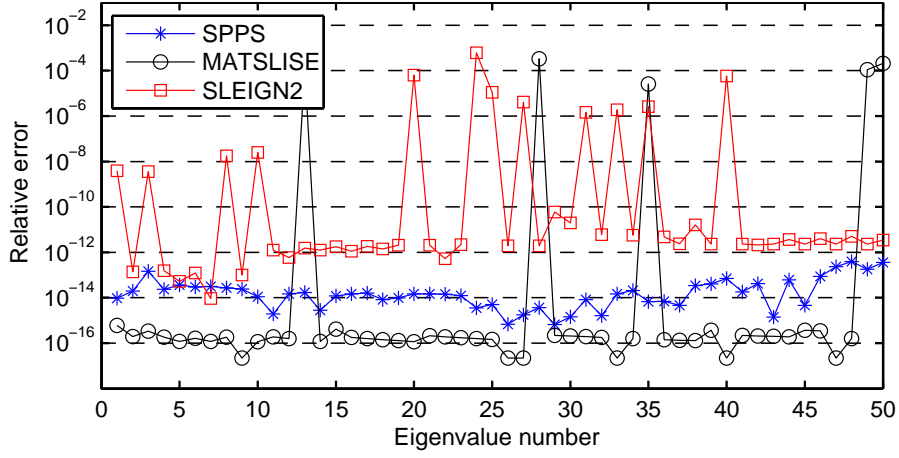}
\caption{Relative error of the first 50 eigenvalues of the spectral problem for the Boyd equation
\eqref{BEZex3eqn} calculated with SPPS, SLEIGN2 and MATSLISE.}%
\label{BEZex3error}%
\end{figure}

The SPPS representation allows one to calculate easily the approximate
eigenfunctions. The eigenfunctions that were obtained for this example are
shown in Fig. \ref{BEZex3fig}.

\begin{figure}[ptb]
\centering
\includegraphics[bb=126 295 486 497,
height=2.8in,
width=5in
]
{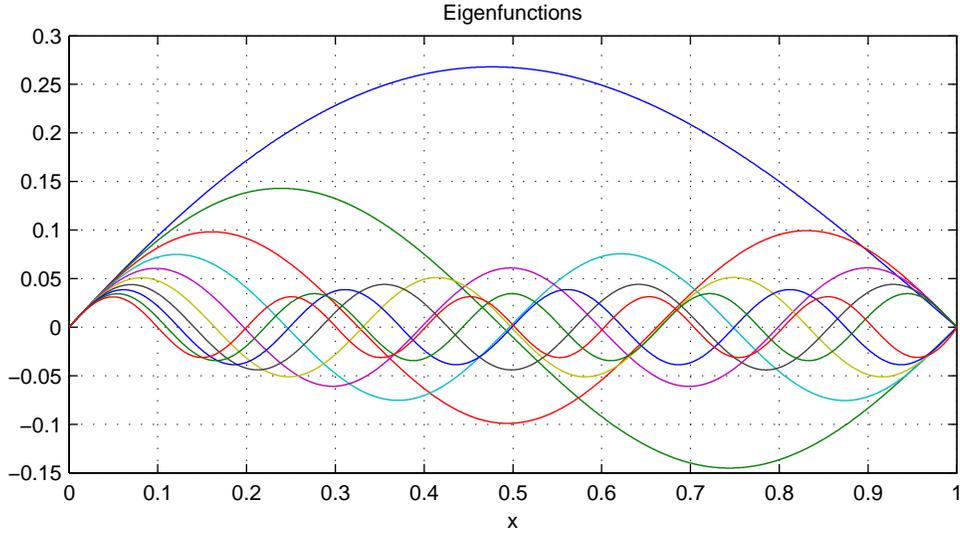}
\caption{The first 10 eigenfunctions of problem \eqref{BEZex3eqn} calculated
with the use of the SPPS method.}%
\label{BEZex3fig}%
\end{figure}
\end{example}

\begin{example}
\label{Ex3} Consider the following spectral problem \cite[Example
2]{BoumenirChanane}.
\begin{equation}
\label{BCex2eqn}%
\begin{cases}
-y^{\prime\prime}+\left(  \frac{\nu^{2}-\frac{1}{4}}{x^{2}}+x^{2}\right)
y=\lambda y, & 0<x\leq\pi\\
y(\pi,\lambda)=0. &
\end{cases}
\end{equation}
The parameter $\nu$ was chosen equal to $2$. In this example we computed both
partial and general solutions using the SPPS representations. Since the exact
partial solution satisfying the asymptotic condition \eqref{SolAsymptotic} is
known and is given by $4\sqrt x I_{1}\bigl(\frac{x^{2}}2\bigr)$, we compared
the approximate solution with the exact one. For $N=40$ and $M=50000$ the
absolute error was less than $7\cdot10^{-16}$. With the aid of Maple software we found
that the characteristic equation of problem \eqref{BCex2eqn} has the form
\begin{equation}
\label{Ex3CharEq}\frac1{\sqrt{\pi}}M_{\lambda/4,1}(\pi^{2})=0,
\end{equation}
where $M_{\lambda/4,1}$ is the Whittaker function, and used equation
\eqref{Ex3CharEq} to compute the exact eigenvalues. In Table \ref{Ex3Table2}
we present the results obtained by the SPPS method, exact eigenvalues and the
results from \cite{BoumenirChanane}, MATSLISE and SLEIGN2. As in Example \ref{Ex1} here again all presented
digits in the exact eigenvalues coincide with the results delivered by MATSLISE. The performance of the SPPS method was considerably better than that of SLEIGN2 and even the 50th eigenvalue was computed correctly to 9 decimal places.

\begin{table}[ptb]
\centering
\begin{tabular}
[c]{cr@{}lr@{}lr@{}lr@{}l}\hline
$n$ & \multicolumn{2}{c}{$\sqrt{\lambda_{n}}$ (SPPS)} &
\multicolumn{2}{c}{$\sqrt{\lambda_{n}}$ (Exact/MATSLISE)} &
\multicolumn{2}{c}{$\sqrt{\lambda_{n}}$ (SLEIGN2)} & \multicolumn{2}{c}{$\sqrt
{\lambda_{n}}$ (\cite{BoumenirChanane})}\\\hline
1 & 2. & 4629499739737 & 2. & 4629499739740 & 2. & 46295003 & 2. & 462949030\\
2 & 3. & 2883529299493 & 3. & 2883529299426 & 3. & 28835311 & 3. & 288339398\\
3 & 4. & 1498642187456 & 4. & 1498642187448 & 4. & 14986471 & 4. & 149833151\\
4 & 5. & 0636688237348 & 5. & 0636688237341 & 5. & 0636695 & 5. & 063634795\\
5 & 6. & 0075814581165 & 6. & 0075814581160 & 6. & 0075836 & 6. & 007577378\\
7 & 7. & 9397373768999 & 7. & 9397373768993 & 7. & 939745 &  & \\
10 & 10. & 8861250916182 & 10. & 8861250916173 & 10. & 886149 &  & \\
15 & 15. & 8426318195682 & 15. & 8426318195682 & 15. & 84275 &  & \\
20 & 20. & 8202301908125 & 20. & 8202301908124 & 20. & 82057 &  & \\
30 & 30. & 7973502195887 & 30. & 7973502195868 & 30. & 7989 &  & \\
50 & 50. & 77867680977 & 50. & 77867680951 & 50. & 789 &  & \\\hline
\end{tabular}
\caption{The eigenvalues from Example \ref{Ex3}, calculated with the use of
the spectral shift technique with $N=50$, $M=50000$. The spectral shift is
given by given by $\lambda_{0}^{(n)}=10n+(n+1)i$.}%
\label{Ex3Table2}%
\end{table}
\end{example}

\begin{example}
\label{Ex4} Consider a particular case of the hydrogen atom equation
\cite[Example 4]{BoumenirChanane}.
\begin{equation}
\label{BCex4eqn}%
\begin{cases}
-y^{\prime\prime}+\left(  \frac{c}{x^{2}}+\frac{1}{x}\right)  y=\lambda y, &
0<x\leq\pi\\
y(\pi,\lambda)=0. &
\end{cases}
\end{equation}
We have chosen $c=6$ and found with the help of Maple the
characteristic equation of problem \eqref{BCex4eqn}
\begin{equation}
\label{Ex4CharEq}\frac i{8\lambda\sqrt{\lambda}}M_{\frac{i}{2\sqrt{\lambda}%
},\frac52}(2i\sqrt{\lambda}\pi)=0,
\end{equation}
where $M_{\frac{i}{2\sqrt{\lambda}},\frac52}$ is the Whittaker function. In
this example we computed both partial and general solutions using the SPPS
representations. The procedure for computing the eigenvalues for this example is
completely analogous to what was described above. In Table \ref{Ex4Table2} we present
the results obtained by the SPPS method, exact eigenvalues and the results
from \cite{BoumenirChanane}, MATSLISE and SLEIGN2. Again all presented digits
in the exact eigenvalues coincide with the results delivered by MATSLISE. The SPPS method and SLEIGN2 performed similarly to the previous example.

The next example shows that the situation may change significantly when another value of the parameter $c$ is chosen.
\end{example}

\begin{example}\label{Ex4b} Consider the same problem as in Example \ref{Ex4} but $c=-1/4$ which corresponds to $l=-1/2$ in \eqref{MainEq}. The computation of eigenvalues for
problem \eqref{BCex4eqn} performed by MATSLISE took several hours on Intel i7-3770 microprocessor meanwhile the computation time required by SLEIGN2 and SPPS did not change significantly. We
present the relative error of the first 50 eigenvalues on Figure
\ref{FigEx4error}. For the SPPS method we used $N=40$ and $M=1000000$. The
spectral shift was given by $\lambda_{0}^{(n)}=10n+(n+1)i$. The accuracy of the results delivered by MATSLISE was considerably lower and once again we emphasize the stability of the accuracy of the eigenvalues computed by the SPPS method in comparison to SLEIGN2. It is worth mentioning that not only for the extreme value $l=-1/2$ but also for the values close to $-1/2$ the computation time required by the MATSLISE package increases significantly.
\end{example}

\begin{table}[ptb]
\centering
\begin{tabular}
[c]{cr@{}lr@{}lr@{}lr@{}l}\hline
$n$ & \multicolumn{2}{c}{$\sqrt{\lambda_{n}}$ (SPPS)} &
\multicolumn{2}{c}{$\sqrt{\lambda_{n}}$ (exact/MATSLISE)} &
\multicolumn{2}{c}{$\sqrt{\lambda_{n}}$ (SLEIGN2)} & \multicolumn{2}{c}{$\sqrt
{\lambda_{n}}$ (\cite{BoumenirChanane})}\\\hline
1 & 1. & 97027445061574 & 1. & 97027445061572 & 1. & 9702743 & 1. &
970274439470\\
2 & 3. & 00436042551872 & 3. & 00436042551857 & 3. & 0043600 & 3. &
004360435708\\
3 & 4. & 01515351791731 & 4. & 01515351791736 & 4. & 015153523 & 4. &
015153641332\\
4 & 5. & 0193472218607 & 5. & 0193472218612 & 5. & 0193459 & 5. &
019347630098\\
5 & 6. & 0210053515482 & 6. & 0210053515488 & 6. & 0210049 & 6. &
021006315094\\
7 & 8. & 0215089715470 & 8. & 0215089715478 & 8. & 0215092 &  & \\
10 & 11. & 0202653559392 & 11. & 0202653559399 & 11. & 0202663 &  & \\
15 & 16. & 0176675547289 & 16. & 0176675547294 & 16. & 0176656 &  & \\
20 & 21. & 0155251794158 & 21. & 0155251794156 & 21. & 0155504 &  & \\
30 & 31. & 0125189152594 & 31. & 0125189152597 & 31. & 01269 &  & \\
50 & 51. & 00916429569 & 51. & 00916429551 & 51. & 0097 &  & \\\hline
\end{tabular}
\caption{The values of $\sqrt{\lambda_{n}}$ from Example \ref{Ex4}, calculated
with the use of the spectral shift technique with $N=40$, $M=50000$. The
spectral shift is given by $\lambda_{0}^{(n)}=10n+(n+1)i$.}%
\label{Ex4Table2}%
\end{table}

\begin{figure}[ptb]
\centering
\includegraphics[bb=126 295 486 497,
height=2.8in,
width=5in
]
{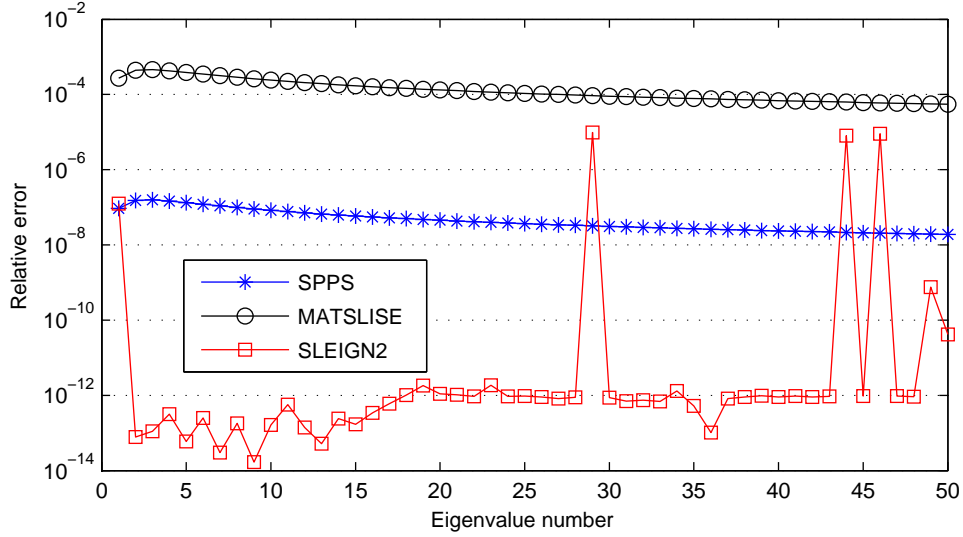}
\caption{Relative error of the first 50 eigenvalues of problem
\eqref{BCex4eqn} with $c=-1/4$ calculated with SPPS, SLEIGN2 and MATSLISE.}%
\label{FigEx4error}%
\end{figure}

\begin{example}
\label{Ex5} Consider the following spectral problem \cite[Example
3]{BoumenirChanane}.
\begin{equation}
\label{BCex3eqn}%
\begin{cases}
-y^{\prime\prime}+\left(  \frac{2}{x^{2}}+\sin x\right)  y=\lambda y, &
0<x\leq\pi\\
y(\pi,\lambda)=0. &
\end{cases}
\end{equation}
For this problem we were unable to find an exact characteristic equation.
We computed both particular and general solutions using the SPPS representations.
The obtained results are presented in Table \ref{Ex5Table1}. The eigenvalues computed by the SPPS method are very close to those delivered by MATSLISE.
\end{example}

\begin{table}[ptb]
\centering
\begin{tabular}
[c]{cr@{}lr@{}lr@{}lr@{}l}\hline
$n$ & \multicolumn{2}{c}{$\sqrt{\lambda_{n}}$ (SPPS)} &
\multicolumn{2}{c}{$\sqrt{\lambda_{n}}$ (MATSLISE)} &
\multicolumn{2}{c}{$\sqrt{\lambda_{n}}$ (SLEIGN2)} &
\multicolumn{2}{c}{Results from \cite{BoumenirChanane}}\\\hline
1 & 1. & 69965392162508 & 1. & 69965392162512 & 1. & 69965496630163 & 1. &
699674822427\\
2 & 2. & 60438727880122 & 2. & 60438727880111 & 2. & 60439344911062 & 2. &
604506077325\\
3 & 3. & 56972957088682 & 3. & 56972957088910 & 3. & 56974717663817 & 3. &
570068095387\\
4 & 4. & 55232022604084 & 4. & 55232022604096 & 4. & 55235767643976 & 4. &
553053525686\\
5 & 5. & 54189892161891 & 5. & 54189892161906 & 5. & 54196736047125 & 5. &
543261224280\\
7 & 7. & 53001773432564 & 7. & 53001773432606 & 7. & 53019003621712 &  & \\
10 & 10. & 5211087141260 & 10. & 5211087141255 & 10. & 5215767716971 &  & \\
15 & 15. & 5141539227764 & 15. & 5141539227760 & 15. & 5156303086973 &  & \\
20 & 20. & 5106568768322 & 20. & 5106568768319 & 20. & 5139986174471 &  & \\
30 & 30. & 5071385063024 & 30. & 5071385063018 & 30. & 5174365365478 &  & \\
50 & 50. & 5043027454211 & 50. & 5043027452760 & 50. & 5434201591972 &  &
\\\hline
\end{tabular}
\caption{The values of $\sqrt{\lambda_{n}}$ from Example \ref{Ex5}, calculated
with the use of the spectral shift technique with $N=40$, $M=50000$. The
spectral shift is given by $\lambda_{0}^{(n)}=10n+(n+1)i$.}
\label{Ex5Table1}%
\end{table}

\subsection{Complex spectrum}
Numerical tests discussed in the previous subsection show that the SPPS method is highly competitive with the best existing codes on their field of applicability. It delivers stable and reliable results even though there remains still plenty of room for improving different computational aspects of the developed programs which implement the SPPS method. Moreover, the range of applicability of the SPPS method to the difference from the other considered codes includes complex coefficients, differential operator on the right-hand side of \eqref{MainEq} and complex eigenvalues. Here we present one such example.

\begin{example}
\label{Ex6} Consider the following spectral problem with the right-hand
side of the equation involving a derivative
\begin{equation}
\label{ExComplex}%
\begin{cases}
-y^{\prime\prime}+\frac{l(l+1)}{x^{2}}y=\lambda y^{\prime}, & 0<x\leq1\\
y^{\prime}(1,\lambda)=0. &
\end{cases}
\end{equation}
Using Maple we found that the solution of equation
\eqref{ExComplex} satisfying the asymptotic condition \eqref{SolAsymptotic} is
given by the expression
\[
y(x;\lambda) = \frac{2^{2l+1}\Gamma(l+3/2)}{\lambda^{l+1/2}}\sqrt
{x}e^{-\lambda x/2}I_{l+1/2}\left(  \frac{\lambda x}2\right)  ,
\]
where $I$ is the modified Bessel function of the first kind. The
derivative of this solution has the form
\[
y^{\prime}(x;\lambda) = -\frac{2^{2l}\Gamma(l+3/2) e^{-\lambda x/2}}%
{\lambda^{l+1/2}\sqrt x}\left(  (2l+\lambda x)I_{l+1/2}\left(  \frac{\lambda
x}2\right)  -\lambda x I_{l-1/2}\left(  \frac{\lambda x}2\right)  \right)  .
\]
For a numerical experiment we have chosen $l=1/2$, hence the exact
characteristic equation is
\begin{equation}
\label{ExComplexCharEq}\Phi(\lambda):=-\frac{2 e^{-\lambda/2}}{\lambda}\left(
(1+\lambda)I_{1}\left(  \frac{\lambda}2\right)  -\lambda I_{0}\left(
\frac{\lambda}2\right)  \right)  =0.
\end{equation}

\begin{figure}[ptb]
\centering
\includegraphics[bb=126 298 486 493,
height=2.7in,
width=5in
]
{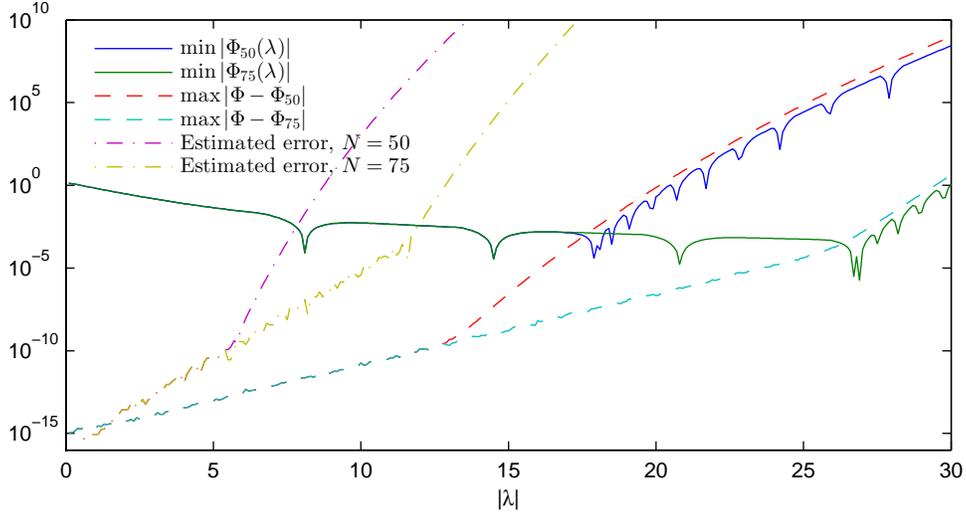}
\caption{The graphs of $\min_{|\lambda|=r}|\Phi_N(\lambda)|$,
$\max_{|\lambda|=r}|\Phi(\lambda)-\Phi_N(\lambda)|$ and estimate \eqref{EqSPPSsolDiscrepancy} as an illustration for Rouch\'{e}'s theorem in Example \ref{Ex6} for $N=50$ and  $N=75$.}
\label{FigComplexRouche}
\end{figure}

\begin{figure}[ptb]
\centering
\includegraphics[
height=3.333in,
width=5.00in,
natwidth=1500,
natheight=1000
]{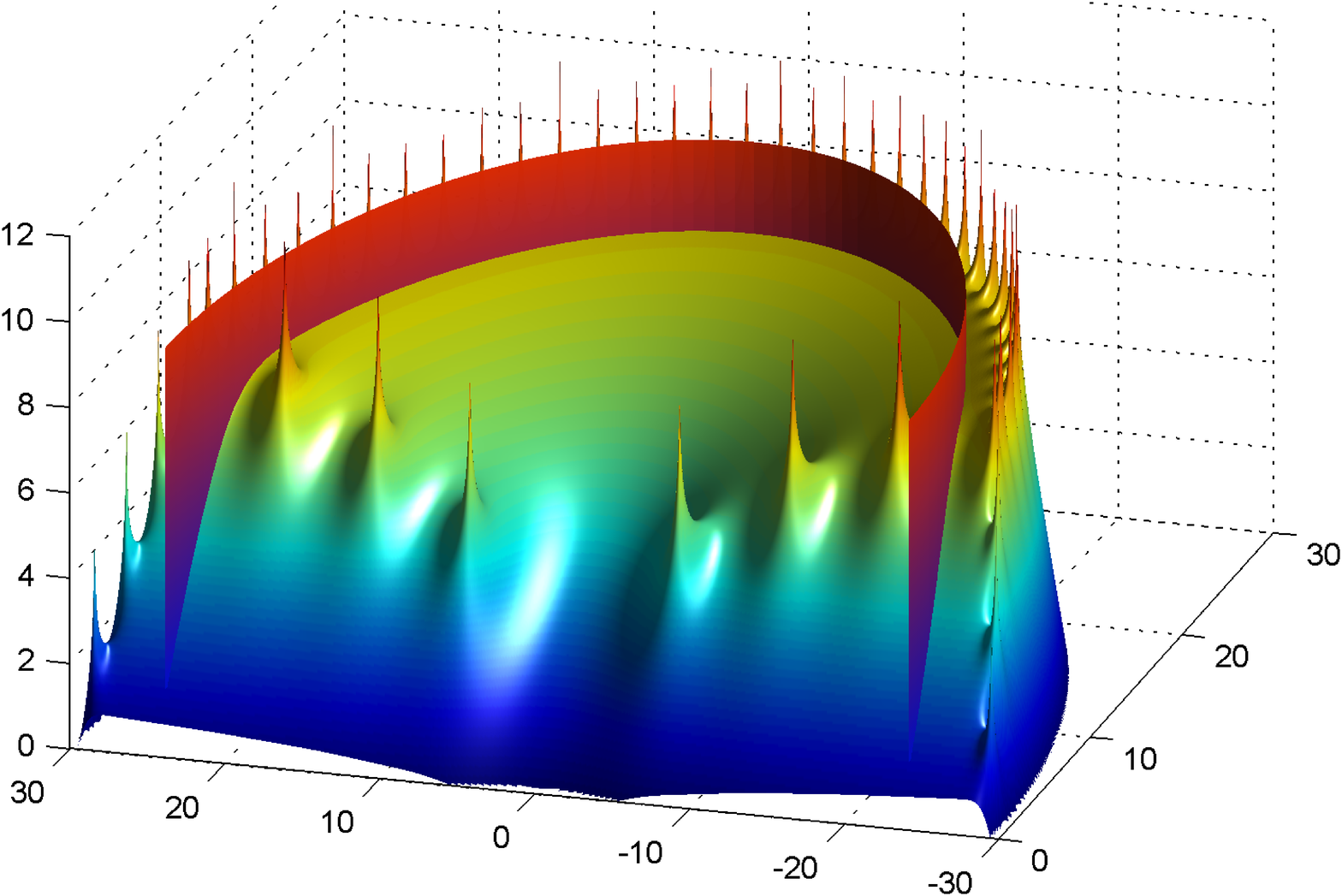}\newline\medskip
\includegraphics[height=2.8in,
width=5.00in,
natwidth=1500,natheight=841
]{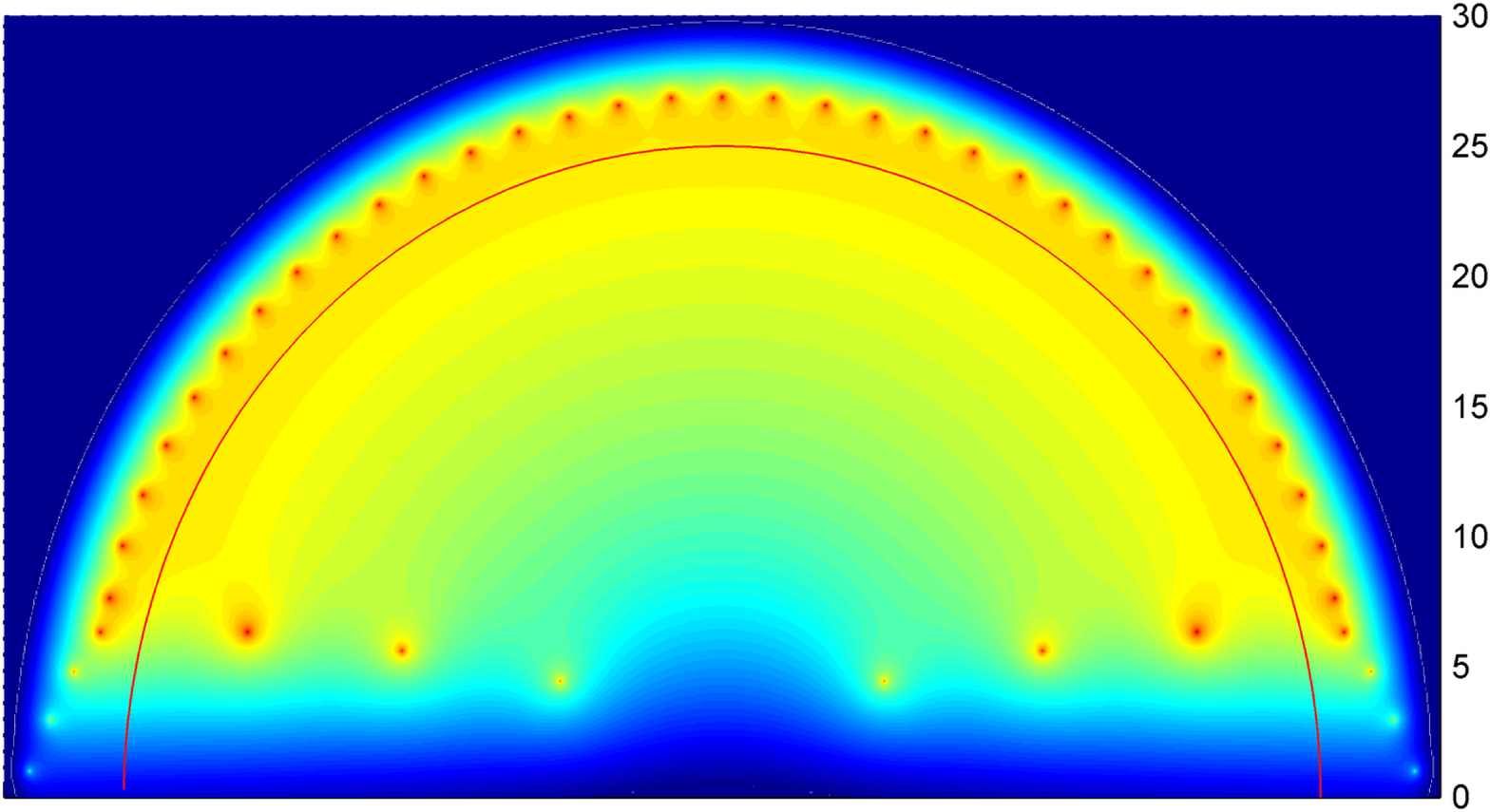}
\caption{The image above corresponds to the graph of $-\log|\Phi_{75}(\lambda)|$. The semicylinder marks the boundary of the disk $|\lambda|=24$ illustrating the region of applicability of  Rouch\'{e}'s theorem, see Example \ref{Ex6}. The six peaks inside the region represent approximate eigenvalues. The rest of the peaks correspond to the roots of the polynomial $\Phi_{75}(\lambda)$ appearing due to the truncation procedure and do not approximate the true eigenvalues of the problem. \newline
\indent The image below is the top view of the one above.}
\label{Ex6fig}
\end{figure}

\begin{table}[ptb]
\centering
\begin{tabular}
[c]{cr@{+}lr@{+}l}\hline
$n$ & \multicolumn{2}{c}{$\lambda_{n}$ (SPPS)} & \multicolumn{2}{c}{$\lambda
_{n}$ (Exact)}\\\hline
1 & 4.47123493365 & 6.76481747492$i$ & 4.47123493371 & 6.76481747480$i$\\
2 & 5.63553225528 & 13.37799928406$i$ & 5.63553225515 & 13.37799928396$i$\\
3 & 6.35749327967 & 19.82515033089$i$ & 6.35749327947 & 19.82515033081$i$\\
4 & 6.88515096019 & 26.20887598267$i$ & 6.88515095992 & 26.20887598266$i$\\
5 & 7.30184486323 & 32.56088281571$i$ & 7.30184486294 & 32.56088281579$i$\\
10 & 8.62739882797 & 64.14303168943$i$ & 8.62739882786 & 64.14303168978$i$\\
20 & 9.98333956705 & 127.0816376255$i$ &  9.98333956726 & 127.0816376257$i$\\
30 & 10.7844002548 & 189.9555609957$i$ & 10.7844002552 &  189.9555609955$i$\\
50 & 11.7983559318 & 315.6569255418$i$ & 11.7983559297 &  315.6569255437$i$\\
75 & 12.6055406455 & 472.7574366522$i$ &  12.6055406452 & 472.7574366509$i$\\
100 & 13.1790674123 & 629.8482784849$i$ & 13.1790674160 &
629.8482784850$i$\\\hline
\end{tabular}
\caption{The values of $\lambda_{n}$ from Example \ref{Ex6}, calculated with
the use of the spectral shift technique, $M=200000$ and $N=50$.}%
\label{Ex6Table2}%
\end{table}

When a problem admits complex eigenvalues we need to distinguish which roots of the polynomial $\Phi_N(\lambda)$ correspond to the eigenvalues and which are spurious roots appearing due to the truncation procedure. Contrary to the case of a purely real spectrum we cannot simply discard all roots whose imaginary part is greater than some $\varepsilon$. Instead, Rouch\'{e}'s theorem and estimate \eqref{EqSPPSsolDiscrepancy} suggest that the roots closest to the origin (or to the current centre $\lambda_0$ when the spectral shift is used) cannot be the spurious roots. We give an illustration of application of this theorem. According to Rouch\'{e}'s theorem we need to find such values of the radius $r$ that
\begin{equation}\label{EqRouche}
\min_{|\lambda|=r}|\Phi_N(\lambda)| > \max_{|\lambda|=r}|\Phi(\lambda)-\Phi_N(\lambda)|,
\end{equation}
establishing that the numbers of zeros of the functions $\Phi_N(\lambda)$ and $\Phi(\lambda)$ coincide inside the disk $|\lambda|<r$.
We calculate $\min_{|\lambda|=r}|\Phi_N(\lambda)|$ using the SPPS representation. To estimate the difference $|\Phi(\lambda)-\Phi_N(\lambda)|$ we use \eqref{EqSPPSsolDiscrepancy}. According to Lemma \ref{LemmaXtildeEstimate}, for problem \eqref{ExComplex} one can take $C=3/2$. We present the obtained results on Figure \ref{FigComplexRouche} for $N=50$ and $N=75$ where it is compared to $\max_{|\lambda|=r}|\Phi(\lambda)-\Phi_N(\lambda)|$ calculated with the aid of the exact characteristic function \eqref{ExComplexCharEq}. As can be seen from the graphs, when we use the exact characteristic function for the estimation of the error the largest values of $r$ for which inequality \eqref{EqRouche} holds are $r_{50}\approx 17$ and $r_{75}\approx 25$.
The estimate \eqref{EqSPPSsolDiscrepancy} delivers rougher estimates, $r'_{50}\approx 7$ and $r'_{75}\approx 12$. On Figure \ref{Ex6fig} we show the graph of $-\log|\Phi_{75}(\lambda)|$ together with the boundary of the disk $|\lambda|=24$. The peaks on the graph correspond to the roots of the polynomial $\Phi_{75}$. Only six of them are located in the disk and hence should be considered as approximations to the eigenvalues. All other roots are located outside the disk and hence should be discarded.

To calculate the higher order eigenvalues, we applied the spectral shift
technique described in Section \ref{SectSpShift}. The following strategy for choosing the values of the spectral shift was implemented. Let $\lambda_{*}^{(n)}$ be the spectral shift on the $n$-th step and $\lambda_{1},\ldots,\lambda_{n}$ the already found eigenvalues. Based on
the representations \eqref{SPPSsolSpShift} and \eqref{SPPSdersolSpShift} we
calculate the roots of the polynomial $\Phi_{N}^{(n)}(\lambda)$ and reorder them
with respect to the distance from $\lambda_{*}^{(n)}$. From these ordered
roots we choose the closest to the point $\lambda_{*}^{(n)}$ and sufficiently distant from $\lambda_{1},\ldots,\lambda_{n}$. We
denote this root as $\lambda_{n+1}$ and set $\lambda_{*}^{(n+1)} =
\lambda_{n+1}+\Delta\lambda$, where $\Delta\lambda$ is a fixed displacement.
Performing this procedure with $\lambda_{*}^{(0)}=0$, $\Delta\lambda=-i$,
$N=50$ and $M=200000$ we obtained the list of eigenvalues presented in Table
\ref{Ex6Table2}. The eigenvalues are ordered according to their distance to the origin. As can be seen from the table the approximate eigenvalues are calculated with a remarkable accuracy.
\end{example}

\section*{Acknowledgements}

We thank our colleague R. Michael Porter for providing us with a first version of a 6 points Newton-Cottes Matlab
numerical integration routine. R. Castillo would like to thank the support of
the SIBE and EDI programs of the IPN as well as that of the project SIP 20120438.
Research of V. Kravchenko and S. Torba was partially supported by
CONACYT, Mexico via the project 166141.

\end{document}